\documentclass[11pt,a4paper]{amsart}
\usepackage{fullpage}
\usepackage{amssymb,amscd,hyperref}
\usepackage[all]{xy}%\CompileMatrices

\newtheorem{thm}{Theorem}[section]
\newtheorem{lem}[thm]{Lemma}
\newtheorem{prop}[thm]{Proposition}

\theoremstyle{definition}
\newtheorem{rem}[thm]{Remark}
\newtheorem{defn}[thm]{Definition}

\DeclareMathOperator{\colim}{colim}
\def\Z{\mathbb{Z}}
\def\Q{\mathbb{Q}}
\def\F{\mathbb{F}}
\def\R{\mathcal{R}}
\def\C{\mathcal{C}}
\def\call{\mathcal{L}}
\def\M{\mathcal{M}}
\def\T{\mathsf{Top}}
\def\THH{\mathsf{THH}}
\def\TAQ{\mathsf{TAQ}}
\def\AQ{\mathsf{AQ}}
\def\Der{\mathsf{Der}}
\def\Hom{\mathsf{Hom}}
\def\Pic{\mathsf{Pic}}
\def\pic{\mathsf{pic}}
\def\PIC{\mathsf{PIC}}
\def\Picard{\mathsf{Picard}}
\def\TMF{\mathsf{TMF}}
\def\Tmf{\mathsf{Tmf}}
\def\tmf{\mathsf{tmf}}
\def\Sp{\mathsf{Sp}}
\def\HR{\mathit{HR}}
\def\HZ{\mathit{H}\!\mathbb{Z}}
\def\HA{\mathit{HA}}
\def\MU{\mathit{MU}}
\def\BP{\mathit{BP}}
\def\KO{\mathit{KO}}
\def\KU{\mathit{KU}}
\def\ko{\mathit{ko}}
\def\ku{\mathit{ku}}

\def\Ext{\mathsf{Ext}}

\def\ra{\rightarrow}
\def\leq{\leqslant}
\def\geq{\geqslant}

\def\call{\mathcal{L}}

\def\HH{\mathsf{HH}}

\def\bfn{\mathbf{n}}
\def\bfm{\mathbf{m}}
\def\bf0{\mathbf{0}}

\def\ie{\emph{i.e.}}

\def\id{\mathrm{id}}

\numberwithin{equation}{section}
\begin{document}
\title{Commutative ring spectra}
\author{Birgit Richter}
\address{Fachbereich Mathematik der Universit\"at Hamburg,
Bundesstra{\ss}e 55, 20146 Hamburg, Germany}
\email{richter@math.uni-hamburg.de}
\urladdr{http://www.math.uni-hamburg.de/home/richter/}
\date{\today}

\keywords{Commutative ring spectra, $E_\infty$-structures on ring spectra,
obstruction theory,  units of ring spectra, Thom spectra, topological
Hochschild homology, topological Andr\'e-Quillen homology, \'etale maps,
descent}
\subjclass[2000]{55P43, 55N35}

\begin{abstract}
In this survey paper on commutative ring spectra we present some basic
features of commutative ring spectra and discuss model category
structures. As a first interesting class of examples of such ring spectra we
focus on (commutative) algebra spectra over commutative
Eilenberg-Mac\,Lane ring spectra. We present two constructions that
yield commutative ring spectra: Thom spectra associated to infinite
loop maps  and Segal's
construction starting with bipermutative categories.  We define topological
Hochschild  homology, some of its variants, and topological
Andr\'e-Quillen homology. Obstruction theory for
commutative structures on ring spectra is
described in two versions. The notion of \'etale extensions in the spectral
world is tricky and we explain why. We define Picard groups and Brauer groups
of commutative ring spectra and present examples.
\end{abstract}
\maketitle

\section{Introduction}

Since the 1990's we have several symmetric monoidal categories of
spectra at our disposal whose homotopy category is the stable homotopy
category. The monoidal structure is usually denoted by $\wedge$ and is
called the smash product of spectra. So since then we can talk about
commutative objects in any of these categories -- these are commutative
ring spectra. Even before such symmetric monoidal categories were
constructed, the consequences of their existence were described. In
\cite[\S 2]{waldhausen} Friedhelm Waldhausen outlines the role of
`rings up to homotopy'. He also coined the expression `brave new
rings' in a 1988 talk at Northwestern.

So what is the problem? Why don't we just write down nice
commutative models of our favorite homotopy types and are down with it?
Why does it make sense to have a whole chapter about this topic?

In algebra, if someone tells you to check whether a given ring is
commutative, then you can sit down and check the axiom for
commutativity and you should be fine. In stable homotopy theory the problem
is more involved, since strict commutativity may only be satisfied by
some preferred point set level model of the underlying associative
ring spectrum and the operadic incarnation of commutativity is an
extra structure rather than a condition.

There is one class of commutative ring spectra that is easy to
construct. If you take singular cohomology 
with coefficients in a commutative ring $R$, then this is represented
by the Eilenberg-MacLane spectrum $\HR$ and this can be represented by
a commutative ring spectrum.

So it would be nice if we could have explicit models for other
homotopy types that come naturally equipped with a commutative ring
structure. Sometimes this is possible. If you are interested in real
(or complex) vector bundles over your space, then you want to
understand real (or complex) topological K-theory and Michael Joachim
\cite{joachim}
for instance produced an explicit analytically flavoured model for periodic
real topological K-theory
as a commutative ring spectrum in the setting of symmetric spectra \cite{HSS}.

There are a few general constructions that produce commutative ring
spectra for you. For instance, the construction of Thom spectra
often gives rise to commutative ring spectra. We will discuss this
important class of examples in Section \ref{sec:thom}. A classical construction
due to Graeme Segal also produces small nice models of commutative ring
spectra (see Section \ref{sec:segal}).

Quite often, however, the spectra that we like are constructed in a
synthetic way: You have some commutative ring spectrum $R$ and you kill a
regular sequence of elements in its graded commutative ring of
homotopy groups, $(x_1,x_2,\ldots)$, $x_i \in \pi_*(R)$, and you
consider a spectrum $E$  with homotopy groups
$\pi_*(E) \cong \pi_*(R)/(x_1,x_2,\ldots)$. Then it is not clear that
$E$ is a commutative ring spectrum.

A notorious example is the Brown-Peterson spectrum, $\BP$. Take the complex
cobordism spectrum $\MU$. Its homotopy groups are
$$ \pi_*(\MU) = \Z[x_1,x_2,\ldots]$$
with $x_i$ being a generator in degree $2i$. If you fix a large even
degree, then you have a lot of possible elements in that degree, so
you might wish to consider a spectrum with sparser homotopy
groups. Using the theory of (commutative, $1$-dimensional) formal
group laws you can do that: If you
consider a prime $p$, then there is a spectrum, called the
Brown-Peterson spectrum, that corresponds to $p$-typical formal group
laws. It can
be realized as the image of an idempotent on $\MU$ and has
$$\pi_*(\BP) \cong \Z_{(p)}[v_1, v_2, \ldots]$$
but now the algebraic generators are spread out in an exponential
manner: The degree of $v_i$ is $2p^i -2$. You can actually choose the $v_i$ as
the $x_{p^i-1}$, so you can think of $\BP$ as a quotient of $\MU$ in the above
sense. Since its birth in 1966 \cite{BP} its
multiplicative properties where an important issue. In \cite{bama} it was
for instance shown that $\BP$ has some partial coherence for homotopy
commutativity, but in 2017 Tyler
Lawson finally shows that at the prime $2$ at least,
$\BP$ is \emph{not} a commutative ring spectrum \cite{lawson}!

There are even worse examples: If you take the sphere spectrum
$S$ and you try to kill the non-regular element $2 \in
\pi_0(S)$ then you get the mod-$2$-Moore spectrum. That isn't
even a ring spectrum up to homotopy. You can also kill all the
generators $v_i \in \pi_*(BP)$ including $p=v_0$ and leaving only one
$v_n$ alive. The resulting spectrum is the connective version of
Morava K-theory, $k(n)$. At the prime $2$ this isn't even homotopy
commutative. In fact, Urs W\"urgler shows more \cite{wuergler}: If $\pi_0$
of a homotopy commutative ring spectrum has characteristic two, then it
is a generalized Eilenberg-Mac\,Lane spectrum. Recent work of Mathew, Naumann 
and Noel puts severe restrictions on
finite $E_\infty$-ring spectra \cite{mnn}.

So how do you decide such questions? How can you determine whether a
given spectrum is a commutative ring spectrum if you don't have a
construction that tells you right
away that it is commutative? This is where obstruction theory comes
into the story.

There is an operadic notion of an $E_\infty$-ring spectrum that goes
back to Boardman-Vogt and May. Comparison theorems \cite{MMSS},
\cite{schwede-comp} then tell you whether these more complicated
objects are equivalent to commutative ring spectra. In the categories
of symmetric spectra, orthogonal spectra and $S$-modules they are.

Obstruction theory might help you with a decision whether a
spectrum carries a commutative monoid structure: One version
\cite{basterra} gives obstructions for lifting the ordinary Postnikov
tower to a Postnikov tower that lives within the category of commutative
ring spectra. The other kind finds some obstruction classes that tell
you that you cannot extend some partial bits and pieces of a nice
multiplication to a fully fledged structure of an $E_\infty$-ring
spectrum or that some homology or homotopy operation that you observe
contradicts such a structure. This can be used for a negative result (as
in \cite{lawson}) or for positive statements: There are result by Robinson
\cite{robinson} and Goerss-Hopkins \cite{GH,GHall} that tell you that
you have a (sometimes even unique) $E_\infty$-ring structure on your
spectrum if all the obstruction \emph{groups} vanish. Most notably
Goerss and Hopkins used obstruction theory to prove that the Morava
stabilizer groups acts on the corresponding Lubin-Tate spectrum via
$E_\infty$-morphisms \cite{GH}.

There are other things that are weird for commutative ring
spectra. Quite often, we end up working with ideals in the graded
commutative ring of homotopy groups, but as we saw above, this is not
a suitable notion of ideal. There is a notion of an ideal in the
context of (commutative) ring spectra \cite{hovey-ideals} due to Jeff
Smith, but still several algebraic constructions do not have an
analogue in spectra.

The algebraic behaviour on the level of homotopy groups can be quite
deceiving: complexification turns a real vector bundle into a complex
vector bundle. This induces a map $\pi_*(\KO) \ra \pi_*(\KU)$ which can
be realized as a map of commutative ring spectra $c \colon \KO \ra \KU$.
On homotopy groups we get
\begin{equation} \label{eq:kotoku}
\pi_*(c) \colon \pi_*(\KO) = \Z[\eta, y, \omega^{\pm 1}]/(2\eta,
\eta^3, \eta y, y^2-4w) \ra \Z[u^{\pm 1}] \pi_*(\KU).
\end{equation}
Here the degrees are $|\eta|=1$, $|y|=4$, $|w|=8$, $|u|=2$ and $y$ is sent
to $2u^2$. So on
the algebraic level $c$ is horrible. But John Rognes showed that
the conjugation action on $\KU$ turns the map $c \colon \KO \ra \KU$
into a $C_2$-Galois extension of commutative ring spectra!

Even for ordinary rings, viewing a (commutative) ring $R$ as a
(commutative) ring spectrum via the Eilenberg-Mac\,Lane spectrum
functor changes the situation completely. The ring $R$ has a characteristic
map $\chi \colon \Z \ra R$ because the integers are the initial ring.
As a ring spectrum, $H\Z$ is far from being initial. The map $H\chi$ can
be precomposed with the unit map of $H\Z$
$$ \xymatrix@1{{S} \ar[r]^\eta & {H\Z} \ar[r]^{H\chi} & {HR}
}$$
and the sphere spectrum $S$ is the initial ring spectrum! Now there
is a lot of space between the
sphere and any ring. I will discuss two consequences that this has:
There is actually algebraic geometry happening between the sphere
spectrum and the
prime field $\F_p$: There is a Galois extension of commutative
ring spectra (see \ref{def:galois}) $A \ra H\F_p$!

Another feature is that there exist differential graded algebras $A_*$
and $B_*$ that are not quasi-isomorphic, but whose associated
algebra spectra over an Eilenberg-Mac\,Lane spectrum
\cite{shipley} are equivalent as ring spectra \cite{DS}.
Similar phenomena happen if you consider differential graded
$E_\infty$-algebras: There are non quasi-isomorphic ones 
whose associated commutative algebras over an Eilenberg-Mac\,Lane
spectrum \cite{rs} are equivalent as commutative ring spectra
\cite{bayindir}.

\subsection*{Content}
The structure of this overview is as follows:
We start with some basic features of commutative ring spectra and
their model category structures in Section \ref{sec:modelcats}. The most
basic way to relate classical algebra to brave new algebra is via the
Eilenberg-Mac\,Lane spectrum functor. We study chain algebras and
algebras over Eilenberg-Mac\,Lane ring spectra in Section
\ref{sec:hralgs}. As you can study the group of units of a ring we
consider units of ring spectra and Thom spectra in Section
\ref{sec:thom}. In Section
\ref{sec:segal} we present a construction going back to Segal. Plugging
in a bipermutative category yields a commutative ring spectrum.

In Section \ref{sec:thhetc} we introduce topological
Hochschild homology and some of its variants and topological
Andr\'e-Quillen homology. 
In Section \ref{sec:obs} we discuss some versions of obstruction theory
that tell you whether a given multiplicative cohomology theory can be
represented by a strict commutative model.

Some concepts from algebra translate directly
to spectra but some others don't. We discuss the different concepts of
\'etale maps
for commutative algebra spectra in Section \ref{sec:etale}. Picard and Brauer
groups for
commutative ring spectra are important invariants and feature in Section
\ref{sec:pic}.

\subsection*{Disclaimers} For more than 30 years, the phrase 
\emph{commutative ring spectrum} meant a commutative monoid in the
homotopy category of spectra. Since the 90's this has changed. At
the beginning of this new era people were careful not to use this name,
in order to
avoid confusion with the homotopy version. In this paper we reserve the
phrase 
\emph{commutative ring spectrum} for a commutative monoid
in some symmetric monoidal category of spectra.

The second disclaimer is that for this paper a space is always compactly
generated weak Hausdorff. I denote the corresponding category just by $\T$.

Last but not least: Of course, this overview is not complete. I had to
omit important aspects of the field due to space constraints. Most
prominently probably is the omission of not discussing topological cyclic
homology and its wonderful applications to algebraic K-theory.

I try to give adequate references, but
often it was just not feasible to describe the whole development of
a topic and much worse, I probably have forgotten to cite important
contributions. If you read this and it affects you, then I can only
apologize.

\subsection*{Acknowledgements}
The  author  thanks  the  Hausdorff  Research  Institute  for
Mathematics in Bonn for its hospitality during the Trimester Program
\emph{K-Theory and Related Fields}. She thanks Akhil Mathew for
clarifying one example, Andy Baker for pointing out an omission and 
 Steffen Sagave for many valuable comments on a draft version of this
 paper. 

\section{Features of commutative ring spectra}
\label{sec:modelcats}
\subsection{Some basics}
Before we actually start with model structures, we state some basic facts
about commutative ring spectra.

Let $R$ be a commutative ring spectrum. Then the category of $R$-module 
spectra is closed symmetric monoidal: For two such $R$-module spectra
$M, N$ the  
smash product over $R$, $M \wedge_R N$, is again an $R$-module and the usual 
axioms of a symmetric monoidal category are satisfied. There is an
$R$-module spectrum $F_R(M, N)$,  the function spectrum of $R$-module
maps from $M$ to $N$. 

We denote the 
category of $R$-module spectra by $\mathcal{M}_R$. The category of commutative
$R$-algebras is the category of commutative monoids in $\mathcal{M}_R$ and we 
denote this category by $\C_R$.

By definition, every object $A$ of $\C_R$ receives a unit map from $R$ and
hence $R$ is initial in $\C_R$. In particular, the sphere spectrum is the
initial commutative ring spectrum. Every discrete ring is a $\Z$-algebra,
similarly, every (commutative) ring spectrum is a (commutative) $S$-algebra.
If $R$ is a commutative ring spectrum, then the category of commutative
$R$-algebra is isomorphic to the category of commutative ring spectra under
$R$, \ie, such commutative ring spectra $A$ with a distinguished map $\eta
\colon R \ra A$ in that category.

We allow the trivial $R$-algebra corresponding to the one-point
spectrum $\ast$ and this spectrum is a terminal object in $\C_R$.

In any symmetric monoidal category $(\C, \otimes, 1, \tau)$ the coproduct
of two commutative monoids $A$ and $B$ in $\C$ is $A \otimes B$. So, for
two commutative $R$-algebras $A$ and $B$, their coproduct is
$A \wedge_R B$.

\subsection{Model structures on commutative monoids}

I will assume that you are familiar with the concept of model categories and
that you have seen some examples and read Chapter 4 in this book. 
Good general reference are Hovey's
\cite{hovey} or Hirschhorn's \cite{hirschhorn} book.
You could also just skip this section and have in mind that there are
some serious model category issues lurking in the dark.

For this section I will mainly focus on two models for spectra: Symmetric
spectra \cite{HSS} and $S$-modules \cite{EKMM}. They are different
concerning their model structures. In the model structure in
\cite{HSS} on symmetric spectra the sphere spectrum is cofibrant
whereas in the one for $S$-modules it is not, but all objects are
fibrant.

The model structures on commutative monoids in either of the categories
\cite{EKMM,HSS} are special cases of a \emph{right induced model structure}:
We have a functor $P_R$ from $R$-module spectra to commutative $R$-algebra
spectra assigning the free commutative $R$-algebra spectrum on $M$ to any
$R$-module spectrum $M$, explicitly
$$ P_R(M) = \bigvee_{n \geq 0} M^{\wedge_R n}/\Sigma_n.$$

The symbol $P_R$ should remind you of a polynomial algebra. 
This functor has a right adjoint, the forgetful functor $U$. In a
right-induced model structure one determines the fibrations and weak
equivalences by the right adjoint functor. In our cases, a map of commutative
$R$-algebra spectra is a fibration or a weak equivalence, if it is one
in the underlying category of $R$-module spectra. Note that establishing
right induced model structures on commutative monoids in some model category
does  not always work. The standard example is the category of $\F_p$-chain
complexes (say $p$ is an odd prime). Then the chain complex $\mathbb{D}^2$
is acyclic, having $\F_p$ in degrees $1$ and $2$ with the identity map as
differential, but the free graded commutative monoid generated by it is
$\Lambda_{\F_p}(x_1) \otimes \F_p[x_2]$ with $|x_i|=i$ and the 
induced differential is determined by $d(x_2)=x_1$ and the Leibniz rule. 
But then $d(x_2^p)$ is a cycle that is not a boundary, so the resulting
object is not acyclic.

If $R$ is a commutative $S$-algebra in the setting of EKMM \cite{EKMM},
then the categories of associative $R$-algebras and of
commutative $R$-algebras possess a right induced model structure
\cite[Corollary VII.4.10]{EKMM}. The existence of
the model structure for commutative monoids is a special case of the
existence of right-induced model structures for $\mathbb{T}$-algebras
(\cite[Theorem VII.4.9]{EKMM}) where $\mathbb{T}$ is a continuous
monad on the category of $R$-module spectra that preserves reflexive
coequalizers and satisfies the cofibration hypothesis
\cite[VII.4]{EKMM}. The category of commutative $S$-algebras is
identified \cite[Proposition II.4.5]{EKMM} with the category of 
algebras for the monad $P_S$ as above on the category of $S$-modules.

In diagram categories such as symmetric spectra and orthogonal spectra
the situation is different: In the standard model structures on these
categories the sphere spectrum is cofibrant. If one would take a
right-induced model structure on the category of commutative monoids,
\ie, the model structure such that a map of commutative ring spectra
$f \colon A \ra B$ is a fibration or weak equivalence if it is one in
the underlying category, then the sphere would still be cofibrant. If
we take a fibrant replacement of the sphere $S \ra S^{\text{fib}}$, then in
particular $S^{\text{fib}}$ would be fibrant in the model category of symmetric
spectra, hence it would be an Omega-spectrum and its zeroth level
would be a strictly commutative model for $QS^0$. However, Moore
shows \cite[Theorem 3.29]{moore} that this implied that $QS^0$ had
the homotopy type of a product of Eilenberg-MacLane spaces --  but
this is false.

The usual way to avoid this problem is to consider a positive model
structure on $Sp^\Sigma$ (see \cite[Definition 6.1]{MMSS} for the general
approach). Here the
positive level fibrations (weak equivalences) are maps $f \in
Sp^\Sigma(X, Y)$ such that $f(n)$ is a fibration (weak equivalence)
for all levels $n \geq 1$. The positive cofibrations are then cofibrations in
$Sp^\Sigma$ that are isomorphisms in level zero. The positive stable
model category is then obtained by a Bousfield localization that forces the
stable equivalences to be the weak equivalences and the right-induced model 
structure on the commutative monoids in $Sp^\Sigma$ then has the desired 
properties. 

There is another nice model for connective spectra, given by $\Gamma$-spaces
\cite{catsandco,lydakis}.
This category is built out of functors from finite pointed sets to spaces,
so it is a very hands-on category with explicit constructions. It is also a
symmetric monoidal category with a suitable model structure.
We refer to \cite{lydakis,schwede-gamma} for background on this. 
Its (commutative) monoids are called (commutative) $\Gamma$-rings.
Beware that commutative $\Gamma$-rings, however, do \emph{not} model all
connective commutative ring spectra. Tyler Lawson proves in
\cite{lawson-gamma} that commutative $\Gamma$-rings satisfy a vanishing
condition for Dyer-Lashof operations of positive degree on classes in
their zeroth mod-$p$-homology (for all primes $p$) and that for instance
the free $E_\infty$-ring spectrum generated by $\mathbb{S}^0$ cannot be
modelled as a commutative $\Gamma$-ring.

\subsection{Behaviour of the underlying modules}

In the setting of EKMM it is shown that the underlying
$R$-modules of  cofibrant commutative $R$-algebras have a well-behaved
smash product in the derived category of $R$-modules:
\begin{thm} \, \cite[Theorem VII.6.7]{EKMM} \,
If $A$ and $B$
are two cofibrant commutative $R$-algebras, and if
$\xymatrix@1{{\varphi_A \colon \Gamma A} \ar[r]^(0.65){\sim} & {A}}$ and
$\xymatrix@1{{\varphi_B \colon \Gamma B} \ar[r]^(0.65){\sim} & {B}}$
are chosen cell $R$-module spectra approximations
then
$$\varphi_A \wedge_R \varphi_B \colon \Gamma A \wedge_R \Gamma B \ra A
\wedge_R B$$
is a weak equivalence.
\end{thm}

Brooke Shipley developed a model structure for commutative symmetric ring
spectra in \cite{shipleyconvenient} in which the
underlying symmetric spectrum of a cofibrant commutative ring spectrum
is also cofibrant as a symmetric spectrum \cite[Corollary
4.3]{shipleyconvenient}.

She starts with introducing a different model structure
on symmetric spectra. Let $M$ denote the class of monomorphisms of
symmetric sequences in pointed simplicial sets and let $S \otimes M$
denote the set $\{S \otimes f, f \in M\}$ where $\otimes$ denotes the
tensor product of symmetric sequences. An
\emph{$S$-cofibration}  is a morphism in $(S \otimes
M)$-cof, \ie, a morphism in $Sp^\Sigma$ that has the left lifting
property with respect to maps that have the right lifting property
with respect to $S \otimes M$. She shows that the classes
of $S$-cofibrations and stable equivalences determine a model structure
with the $S$-fibrations being the class of morphisms with the right
lifting property with respect to $S$-cofibrations that are also stable
equivalences \cite[Theorem 2.4]{shipleyconvenient}. This model
structure was already mentioned in \cite[5.3.6]{HSS}. Shipley proves
that this model structure is cofibrantly generated, is monoidal and
satisfies the monoid axiom \cite[2.4, 2.5]{shipleyconvenient}.

Note
that symmetric spectra are $S$-modules in symmetric sequences. This
allows for a version of an \emph{$R$-model structure} for every associative
symmetric ring spectrum $R$ with $R$-cofibrations, $R$-fibrations and
stable equivalences \cite[Theorem 2.6]{shipleyconvenient}.
In the positive variant of this model structure the positive
$R$-cofibrations are $R$-cofibrations that are isomorphisms in level
zero. Together with the stable equivalences this determines the
\emph{positive $R$-model structure}.

The corresponding right induced model structure on commutative
$R$-algebra spectra for a commutative symmetric ring spectrum $R$ is
then the \emph{convenient} model structure: The weak equivalences are
stable equivalences, the fibrations are positive $R$-fibrations and
the cofibrations are determined by the structure.

She then shows the remarkable property of this model structure on
commutative $R$-algebra spectra:

\begin{thm} \, \cite[Corollary 4.3]{shipleyconvenient} \,
If $A$ is cofibrant as a commutative $R$-algebra then $A$ is
$R$-cofibrant in the $R$-model structure. If $A$ is fibrant
as a commutative $R$-algebra, then $A$ is fibrant 
  in the positive $R$-model structure on $R$-module spectra.
\end{thm}

The positive $R$-model structure ensures that $R$ is \emph{not}
cofibrant, hence cofibrant commutative $R$-algebras will not be
positively $R$-cofibrant!

\subsection{Comparison, rigidification and $E_n$-structures}

Stefan Schwede proves \cite[Theorem 5.1]{schwede-comp} that the
homotopy category
of commutative $S$-algebras from \cite{EKMM} is equivalent to the
homotopy category of commutative symmetric ring spectra by
establishing a Quillen equivalence between the corresponding model
categories. In \cite[Theorem 0.7]{MMSS} the analogous comparison
result is proven for commutative orthogonal ring spectra and
commutative symmetric ring spectra.

Even before any symmetric monoidal category of spectra was
constructed, the notion of operadically defined $E_\infty$-ring
spectra \cite{may-einfty} was available. An $E_\infty$-structure on a
spectrum is a multiplication that is homotopy commutative in a
coherent way. See Chapter 5 of this book
for background on operads and their role in the study of spectra with
additional structure.

There is an explicit comparison of the good old
$E_\infty$-ring spectra and commutative ring spectra, see
\cite[Proposition II.4.5]{EKMM} or \cite[Remark 0.14]{MMSS}, in particular,  
every $E_\infty$-ring spectrum $\tilde{R}$ can be rigidified to a commutative 
ring spectrum $R$ in such a way that the homotopy type is preserved. 

There are several popular $E_\infty$-operads that will show up later: for 
instance the linear isometries operad (see \eqref{eq:call}) and 
the Barratt-Eccles operad. The $n$-ary part of the latter is easy to describe: 
You take $\mathcal{O}(n) = E\Sigma_n$, a contractible space with free $\Sigma_n$-action. 
For compatiblity reasons it is advisible to take the realization of the 
standard simplicial model of $E\Sigma_n$ whose set of $q$-simplices is 
$(\Sigma_n)^{q+1}$. 

An operad with a nice geometric description is the little $m$-cubes operad, 
that in arity $n$ consists of the space of $n$-tuples of  
linearly embedded $m$-cubes in the standard $m$-cube with disjoint interiors 
and with axes parallel to 
that of the ambient cube. \cite[Example 5]{bv}. We call this (and every 
equivalent) operad in spaces $E_m$. For $m=1$ this operad parametrizes 
$A_\infty$-structures and the colimit is an $E_\infty$-operad. Hence the 
intermediate $E_m$'s for $1 < m < \infty$ interpolate between these 
structures, they give $A_\infty$-structures with homotopy-commutative 
multiplications that are coherent up to some order. 
\subsection{Power operations}
The extra structure of an $E_\infty$-ring spectrum gives homology
operations. The general setting allows for $H_\infty$-ring spectra
\cite{hinfty}; for simplicity we assume that $E$ and $R$ are two
$E_\infty$-ring spectra whose structure is given by the Barratt-Eccles
operad, \ie, there are structure maps 
\begin{equation} \label{eq:extendedpowers}
\xi^n_R \colon (E\Sigma_n)_+ \wedge_{\Sigma_n} R^{\wedge n} \ra R
\end{equation}
for $R$ and also for $E$. McClure describes the general setting of
power operations in \cite[IX \S 1]{hinfty}. Fix a prime $p$ and
abbreviate $(E\Sigma_p)_+ \wedge_{\Sigma_p} R^{\wedge p}$ by
$D_p(R)$; $D_pR$ is often known as the \emph{$p$th extended power
  construction on $R$}. A \emph{power operation} assigns to every
class $[x] \in E_nR$ and every class $e \in E_m(D_pS^n)$ a class 
$Q^e[x] \in E_mR$, hence we can view $Q^e$ as a map 
$$ Q^e\colon E_nR \ra E_mR.$$
The construction is as follows. Take a representative $x \colon S^n
\ra E \wedge R$ of $[x]$ and $e \in E_m(D_pS^n)$ and apply the following 
composition to $e$: 
\begin{equation} \label{eq:powerops}
\xymatrix{
{E_m(D_pS^n)} \ar@{.>}[rrdddd] \ar[rr]^{E_m(D_px)} & &  {E_m(D_p(E
  \wedge R))} \ar[d]^{\delta} 
\\
& & {E_m(D_pE \wedge D_pR)} \ar[d]^{E_m(\xi^p_E \wedge \id)}\\
& & {E_m(E \wedge D_pR)} \ar[d]^{\mu_*}\\
& & {E_m(D_pR)} \ar[d]^{E_m(\xi^p_R)}\\
& & {E_m(R).}
}
\end{equation}
Here, 
$$ \delta \colon (E\Sigma_p)_+ \wedge_{\Sigma_p} (E \wedge R)^{\wedge
  p} \ra (E\Sigma_p)_+ \wedge_{\Sigma_p} E^{\wedge p} \wedge 
(E\Sigma_p)_+ \wedge_{\Sigma_p} R^{\wedge p}$$
is the canonical map induced by the diagonal on the space $E\Sigma_p$
and $\mu$ denotes the multiplication in $E$, so it induces
$$ \mu_* \colon \pi_*(E \wedge E \wedge D_pR) \ra \pi_*(E \wedge
D_pR).$$
  
There are several important special cases of this construction: 
\begin{enumerate}
\item
For $E$ the sphere spectrum one obtains operations on the homotopy
groups of an $E_\infty$-ring spectrum, see \cite[IV \S 7]{hinfty}. 
\item
For $E = H\F_p$ the power operations for certain classes $e_i \in
H_i(\Sigma_p; \F_p)$ are often called
(Araki-Kudo-)Dyer-Lashof operations. 
These are natural homomorphisms 
\begin{equation} \label{eq:dl} 
Q^i\colon (H\F_p)_n(R) \ra
(H\F_p)_{n +2i(p-1)}(R) 
\end{equation}
for odd primes and  $Q^i\colon (H\F_2)_n(R) \ra
(H\F_2)_{n + i}(R)$ at the prime $2$ that satisfy a list of axioms
\cite[Theorem III.1.1]{hinfty} and compatibility relations with the
homology Bockstein and the dual Steenrod operations.
\item
There are also important $K(n)$-local versions of such operations and
we will encounter them later.  
\end{enumerate}

\section{Chain algebras and algebras over Eilenberg-Mac\,Lane spectra}
\label{sec:hralgs}
The derived category of a ring is an important object in many
subjects. The initial ring is the ring of integers. Every ring $R$ has
an associated Eilenberg-Mac\,Lane spectrum, $\HR$.

\subsection{$\HR$-module and algebra spectra}
We collect some results that compare
the category of chain complexes of $R$-modules with the category of
module spectra over $\HR$. We start with additive statements and move
to comparison results for flavors of differential graded
$R$-algebras. For an overview of algebraic applications of these
equivalences see for instance \cite{greenlees}.

In the eighties, so before any
strict symmetric monoidal category of spectra was constructed, Alan
Robinson developed the notion of the derived category, $\mathcal{D}(E)$, 
of right 
$E$-module spectra for every $A_\infty$-ring spectrum $E$. He showed
the following result.
\begin{thm} \, \cite[Theorem 3.1]{roderived} \,
For every associative ring $R$ there is an equivalence of categories
between the derived category of $R$, $\mathcal{D}(R)$, and the derived
category of the associated Eilenberg-Mac\,Lane spectrum,
$\mathcal{D}(\HR)$.
\end{thm}

Later, in the context of $S$-modules this corresponds to \cite[IV,
Theorem 2.4]{EKMM}. Work of Schwede and Shipley
strengthened the result to a Quillen equivalence of the corresponding model
categories:
\begin{thm} \, \cite[Theorem 5.1.6]{schwede-shipleydgmod} \,
The model category of unbounded chain complexes of $R$-modules is Quillen
equivalent to the model category of $\HR$-module spectra.
\end{thm}

Stefan Schwede uses the setting of $\Gamma$-spaces \cite{schwede-gamma} to
embed simplicial rings and modules into the stable world: He constructs a lax
symmetric monoidal Eilenberg-Mac\,Lane functor $H$ from simplicial abelian
groups to $\Gamma$-spaces together with a linearization functor $L$ in the
opposite direction and proves the following comparison result:
\begin{thm} \, \cite[Theorems 4.4, 4.5]{schwede-gamma} \,
If $R$ is a simplicial ring, then the adjoint functors $H$ and $L$
constitute a Quillen equivalence between the categories of simplicial
$R$-modules and $\HR$-module spectra. If $R$ is in addition commutative, then 
$H$ and $L$ induce a Quillen equivalence between the categories of simplicial
$R$-algebras and $\HR$-algebra spectra.
\end{thm}
Here, the functor $L$ is actually left inverse to $H$ and induces an 
isomorphism of $\Gamma$-spaces 
$$ \Hom(HA, HB) \cong H(\Hom_{\mathsf{sAb}}(A, B))$$ 
\cite[Lemma 2.1]{schwede-gamma}, thus $H$ embeds algebra into brave new 
algebra.

Brooke Shipley extends this equivalence to corresponding categories
of monoids in the differential graded setting:
\begin{thm} \, \cite[Theorem 1.1]{shipley} \,
For any commutative ring $R$, the model categories of
unbounded differential graded $R$-algebras and $\HR$-algebra spectra are
Quillen equivalent.
\end{thm}

Dugger and Shipley show in \cite{DS} that there are examples of
$\HR$-algebras that are weakly
equivalent as $S$-algebras, but that are not quasi-isomorphic.  A concrete
example is the differential graded ring $A_*$ which is generated by an
element in degree $1$, $e_1$, and has $d(e_1) = 2$ and satisfies $e_1^4=0$.
The corresponding $\HZ$-algebra spectrum is equivalent as a ring spectrum
to the one on the exterior algebra 
$B_* = \Lambda_{\F_2}(x_2)$ (with $|x_2|=2$) but $A_*$ and
$B_*$ are \emph{not}
quasi-isomorphic. You find more examples and proofs in \cite[\S\S 4,5]{DS}.

We cannot expect that commutative $\HR$-algebra spectra correspond to
commutative differential graded $R$-algebras unless $R$ is of characteristic
zero, because of cohomology operations, but we get the following result:

\begin{thm} \, \cite[Corollary 8.3]{rs} \,
Let $R$ be a commutative ring. Then there is a chain of Quillen
equivalences between the model category of commutative 
$\HR$-algebra spectra and $E_\infty$-monoids in the category of
unbounded $R$-chain complexes. 
\end{thm}

Haldun \"Ozg\"ur Bay{\i}nd{\i}r shows \cite{bayindir} that one
can find $E_\infty$-differential graded algebras that are not
quasi-isomorphic, but whose corresponding commutative $\HR$-algebra spectra are
equivalent as commutative ring spectra.

\subsection{Cochain algebras}
A prominent class of examples of commutative $\HR$-algebra spectra
consists of function spectra $F(X_+, \HR)$. Here, $X$ is an
arbitrary space and $R$ is a commutative ring. The diagonal 
$\Delta \colon X \ra X \times X$ and the
multiplication on $\HR$, $\mu_{\HR}$, induce a multiplication
$$ \xymatrix{
F(X_+, \HR) \wedge F(X_+, \HR) \ar[r]  & F(X_+ \wedge X_+, \HR \wedge
\HR) \cong F((X \times X)_+, \HR \wedge \HR) \ar[d]^{\Delta^*, \mu_{\HR}} \\
 &  F(X_+, \HR)
} $$
that turns $F(X_+, \HR)$ into a $\HR$-algebra spectrum.
As the diagonal is cocommutative and as $\mu_{\HR}$ is commutative, the
resulting multiplication is commutative.

These function spectra are models for the singular cochains of a space
$X$ with coefficients in $R$:
$$ \pi_*(F(X_+, \HR)) \cong H^{-*}(X; R).$$
Beware that the homotopy groups of $F(X_+, \HR)$ are concentrated in
non-positive degrees -- \ie, $F(X_+, \HR)$ is coconnective.

Studying the singular cochains of a space $S^*(X; R)$ as a
differential graded
$R$-module is not enough in order to recover the homotopy
type of $X$. If we work over the rational numbers, then Quillen showed
that rational homotopy theory is algebraic in the sense that one can
use rational differential graded Lie algebras or coalgebras as models
for rational homotopy theory \cite{Q-qhtp}. Sullivan
\cite{sullivan} constructed a functor, assigning a rational
differential graded commutative algebra to a space, that is closely
related to the singular cochain functor with rational
coefficients. He used this to classify rational homotopy types.

For a general commutative ring $R$, the singular cochains are an
$E_\infty$-algebra. Mike Mandell proves \cite[Main Theorem]{mandell-cochains}
that the singular cochain
functor with coefficients in an algebraic closure  of $\F_p$,
$\bar{\F}_p$, induces
an equivalence between the homotopy category of connected $p$-complete
nilpotent spaces of finite $p$-type and a full subcategory of the
homotopy category of $E_\infty$-$\bar{\F}_p$-algebras. He also
characterizes those  $E_\infty$-$\bar{\F}_p$-algebras that arise as
cochain algebras of $1$-connected $p$-complete spaces of finity
$p$-type explicitly \cite[Characterization
Theorem]{mandell-cochains}. There is also an integral version of this
result, stating that finite type nilpotent spaces are weakly equivalent if and
only if their $E_\infty$-algebras of integral cochains are
quasi-isomorphic \cite[Main Theorem]{mandell-integral}.

\section{Units of ring spectra and Thom spectra} \label{sec:thom}
One construction that can give rise to highly structured
multiplications on a spectrum is the Thom spectrum construction:
For instance, complex bordism, $MU$, obtains a commutative ring structure this
way. Early on Mahowald emphasized \cite{mahowald} that multiplicatve
properties   
of the structure maps for Thom spectra translate to multiplicative structures 
on the resulting Thom spectra. Their properties and the corresponding 
orientation theory is systematically studied in
\cite{may-einfty}. There is the following general result by Lewis: 
\begin{thm} \,  \cite[Theorem IX.7.1 and Remark IX.7.2]{lms} \,
Assume that $f$ is a map of spaces from $X$ to the classifying space for stable
spherical fibrations, $BG$, that is a $\mathcal{C}$-map for some
operad $\mathcal{C}$ over the linear isometries operad. 
Then the Thom spectrum $M(f)$ associated to $f$ carries
a $\mathcal{C}$-structure.
In particular, infinite loops maps from $X$ to $BG$ give
rise to $E_\infty$-ring spectra.
\end{thm}

Note that $BG$ is the classifying space of the units of the sphere
spectrum, $GL_1(S)$. A naive definition of $GL_1(S)$ is given by the
pullback of the diagram
\begin{equation} \label{eq:GL1S}
\xymatrix{
{GL_1S} \ar@{.>}[r] \ar@{.>}[d] & {\Omega^\infty S} \ar[d] \\
{\pi_0(S)^\times = \{ \pm 1\}} \ar[r] & {\pi_0(S) = \mathbb{Z},}
}
\end{equation}
so by the components of $QS^0$ corresponding to $\pm 1 \in \Z$.

In the following we give a short overview of Thom spectra that arise
in a more general context, where the target of the map is the space of
units, $GL_1(R)$,  for a commutative ring spectrum $R$. The first idea is
to define the space $GL_1(R)$ as the space that represents the functor that
sends a space $X$ to the units in $R^0(X)$. Copying the definition from
\eqref{eq:GL1S} above with $S$ replaced by $R$ gives a valid first definition
of $GL_1(R)$ and it was shown \cite{may-einfty} that for commutative $R$
this model is an $E_\infty$-space.

In the approaches \cite{abghr} and \cite{BSS}, the idea is to replace the
naive model of $GL_1(R)$ with its $E_\infty$-structure with a strictly
commutative model. As spaces with an $E_\infty$-structure are not equivalent
to strictly commutative spaces (that's the problem again that then $QS^0$
would be a
product of Eilenberg-Mac\,Lane spaces \cite{moore}), one has to find a different
category with the property that there is a Quillen equivalence between
commutative monoids in that category and $E_\infty$-monoids in spaces and such
that there are models of $\Omega^\infty(R)$ and $GL_1(R)$ in this category.

In \cite{abghr} the authors work with $\ast$-modules 
and in \cite{BSS} the authors use Schlichtkrull's model of $GL_1(R)$ in
commutative $I$-spaces where $I$ is the skeleton of the category of
finite sets and injections.

The idea is to construct a spectrum version of the assembly map for discrete
rings: If $R$ is a discrete ring and if $R^\times$ is its group of units,
then there is a canonical map
\begin{equation} \label{eq:assembly} \Z[R^\times] \ra R \end{equation}
from the group ring $\Z[R^\times]$ to $R$ that takes an element
$\sum_{i=1}^n a_i r_i$ of $\Z[R^\times]$ (with $a_i \in \Z$ and
$r_i \in R^\times$) to the same sum, but now we use the ring structure of $R$
to convert the formal sum into an actual sum $\sum_{i=1}^n a_i r_i \in R$.
Note that $R^\times$ is an abelian group if $R$ is a commutative ring.

We will sketch both constructions of Thom spectra and briefly discuss the
application in \cite{abghr} to the question when a Thom spectrum allows for
an $E_\infty$-map to some other $E_\infty$-ring spectrum, for instance,
whether one 
can realize an $E_\infty$-version of the string orientation $MO\langle
8\rangle \ra \tmf$ \cite{ahr} or an $E_\infty$-version of a complex orientation
\cite{hopkins-lawson}.  

The focus in \cite{BSS} is on multiplicative properties of the Thom spectrum
functor and on applications to topological Hochschild homology. We present
the results about multiplicative structures and discuss their results on
$\THH$ of Thom spectra in Section \ref{sec:thhetc}.
We'll also describe how the concept of $I$-spaces can be generalized to a
setting in which the units can be adapted to non-connective ring spectra.

\subsection{Thom spectra via $\mathbb{L}$-spaces and orientations}
Fix a countably infinite-dimensional real vector space $U$ and consider
$$ \mathbb{L} = \call(1) = \call(U, U), $$
the space of linear isometries from $U$ to itself. The notation $\call(1)$ is
due to the fact that $\call(1)$ is the $1$-ary part of the famous linear
isometries operad \cite[\S 1]{bv} with arity $n$ term 
\begin{equation} \label{eq:call}
\call(n) = \call(U^n, U).
\end{equation}
See \cite{bv} or \cite{abghr} for details. Note that $\mathbb{L}$ is a 
monoid with respect to composition. 

\begin{defn}
The \emph{category of $\mathbb{L}$-spaces, $\T[\mathbb{L}]$}, is the
category of spaces with a left action of the monoid $\mathbb{L}$.
\end{defn}

Using the $2$-ary part of the linear isometries operad, one can manufacture
a product on $\T[\mathbb{L}]$:
For objects $X, Y$ of $\T[\mathbb{L}]$ their product $X \times_\mathbb{L} Y$ is
the coequalizer
$$\xymatrix@1{{\call(2) \times (\call(1) \times \call(1)) \times X \times Y}
\ar@<0.5ex>[r] \ar@<-0.5ex>[r] &
{\, \call(2) \times X \times Y} \ar@{.>}[r] & {X \times_\mathbb{L} Y.} }$$
Here, one map uses the $\call(1)$-action on the spaces $X$ and $Y$ and
the other map uses the operad product $\call(2) \times \call(1) \times
\call(1) \ra \call(2)$.

As $\call(2) = \call(U^2, U)$ has a left $\call(1)$-action, $X 
\times_\mathbb{L} Y$ is an $\call(1)$-space. The product is associative 
and has a symmetry, but it is only weakly unital. See \cite[\S 4]{bcs} 
for a careful discussion. 

By \cite[Proposition 4.7]{bcs} there is an isomorphism of categories between
commutative monoids with respect to $\times_\mathbb{L}$ and $E_\infty$-spaces
whose $E_\infty$-structure is parametrized by the linear isometries operad.

For strict unitality, one restricts to the full subcategory $\M_*$ of objects
of $\T[\mathbb{L}]$ for which the unit map is a homeomorphism. Such
objects are called \emph{$\ast$-modules}. The commutative
monoids in $\M_*$ again model $E_\infty$-spaces \cite[Proposition 4.11]{bcs}.

For an associative ring spectrum $R$, there is a strictly associative
model in $\M_*$ of the space of units $GL_1(R)$ and the functor $GL_1$
is right adjoint to the inclusion of grouplike objects. One can form a bar
construction, $B_{\times_\mathbb{L}}$, 
of a cofibrant replacement of $GL_1(R)$, $GL_1(R)^c$,
with respect to the monoidal product $\times_\mathbb{L}$, where
$B_{\times_\mathbb{L}}(GL_1(R)^c)$ is the geometric realization of the
simplicial $\M_*$ object 
$$ [n] \mapsto * \times_\mathbb{L}
\underbrace{GL_1^c(R) \times_\mathbb{L} \ldots \times_\mathbb{L} GL_1^c(R)}_n
\times_\mathbb{L} *.$$
Similarly, $E_{\times_\mathbb{L}}(GL_1(R)^c)$ is constructed out of the
simplicial object
$$ [n] \mapsto * \times_\mathbb{L}
\underbrace{GL_1^c(R) \times_\mathbb{L} \ldots \times_\mathbb{L}
GL_1^c(R)}_{n+1}.$$
Adapted to the situation there are suspension spectrum and underlying infinite
loop space functors \cite[Lemma 7.5]{lind}
\begin{equation} \label{eq:loopsusp}
\xymatrix@1{
{\M_*} \ar@<0.5ex>[rrr]^{(\Sigma^\infty_\mathbb{L})_+} & & & {\M_S}
\ar@<0.5ex>[lll]^{\Omega^\infty_S}
}
\end{equation}
that are a Quillen adjoint pair of functors. Here, the suspension functor is
strong symmetric monoidal and the underlying loop space functor is lax
symmetric monoidal.

The spectrum version of the assembly map from \eqref{eq:assembly} is
$$ (\Sigma^\infty_\mathbb{L})_+ (GL_1^c(R)) \ra
(\Sigma^\infty_\mathbb{L})_+ (GL_1(R)) \ra R$$
where the first map comes from the cofibrant replacement of the units and
the second one is the counit of an adjunction \cite[(3.1)]{abghr}.

\begin{defn} \, \cite[Definition 3.12]{abghr} \, \label{def:thomass}
The \emph{Thom spectrum for a map $f \colon X \ra
B_{\times_\mathbb{L}}(GL_1^c(R))$ in
$\M_*$} is the $R$-module spectrum (in the world of \cite{EKMM})
\begin{equation} \label{eq:thomass}
M(f) = (\Sigma^\infty_{\mathbb{L}})_+ P^c
\wedge_{(\Sigma^\infty_{\mathbb{L}})_+ GL_1^c(R)} R.
\end{equation}
\end{defn}
Here, $P^c$ is a cofibrant replacement as a right $GL_1^c(R)$-module
of the pullback
$$ \xymatrix{
{P} \ar[r] \ar[d] & {E_{\times_\mathbb{L}}(GL_1^c(R))} \ar[d] \\
{X} \ar[r] & {B_{\times_\mathbb{L}}(GL_1^c(R))}
}$$
\begin{rem}
Because of the cofibrancy of $P^c$, the smash product in
\eqref{eq:thomass} is actually a derived
smash product. See \cite[\S 3]{abghr} for the necessary background on the
model structures involved.
\end{rem}

In the commutative case, \cite[\S 4, \S 5]{abghr} is set in the classical
setting of $E_\infty$ ring spectra and $E_\infty$-spaces as in
\cite{may-einfty}. If $R$ is a commutative ring spectrum or an $E_\infty$ ring
spectrum then the naive space of units, $GL_1(R)$, is a group-like
$E_\infty$-space and hence is an infinite loop space that has an associated
connective spectrum, $gl_1(R)$,  with $\Omega^\infty gl_1(R) = GL_1(R)$.

The crucial ingredient in this case is the pair of functors
$(\Sigma^\infty_+\Omega^\infty, gl_1)$ that is an adjunction between the
homotopy category of connective spectra and the homotopy category of
$E_\infty$-ring spectra in the sense of Lewis-May-Steinberger.

In particular, one gets a version of the assembly map from
\eqref{eq:assembly}
$$ \Sigma^\infty_+\Omega^\infty(gl_1(R)) \ra R$$
for every $E_\infty$-ring spectrum. By \cite{EKMM} one can replace
$E_\infty$-ring spectra with commutative $S$-algebras, \ie, with commutative
ring spectra. This simplifies the discussion of pushouts and allows us to
replace $\Sigma^\infty_+\Omega^\infty$ by
$(\Sigma^\infty_\mathbb{L})_+ \Omega_S^\infty$ from \eqref{eq:loopsusp} to get
$$ (\Sigma^\infty_\mathbb{L})_+ \Omega_S^\infty(gl_1(R)) \ra R$$
Note, that a map of infinite loop spaces
$f\colon B \ra BGL_1(R)$ encodes the same data as a map of spectra
$f \colon b \ra bgl_1(R)$ where the lower case letters denote the associated
connective spectra. As before we consider the pullback $p$
$$ \xymatrix{
{p} \ar[r] \ar[d] & {egl_1(R)} \ar[d] \\
{b} \ar[r]^(0.4)f & {bgl_1(R)}
}$$
and form the corresponding derived smash product:
\begin{defn}
Let $f \colon b \ra bgl_1(R)$ be a map of connective spectra. Then the
\emph{Thom spectrum associated to $f$, $M(f)$}, is the homotopy pushout in the
category of commutative $S$-algebras
$$
M(f) = (R \wedge (\Sigma^\infty_\mathbb{L})_+ \Omega_S^\infty p)
\wedge^L_{R \wedge (\Sigma^\infty_\mathbb{L})_+ \Omega_S^\infty gl_1(R)} R
$$
\end{defn}
As the (homotopy) pushout is the (derived) smash product, this
resembles the construction from \eqref{eq:thomass}

In the commutative ring spectrum setting the question about orientations is
the following problem: Assume that there is a map of commutative ring spectra
$\alpha \colon R \ra A$, then $A$ is a commutative $R$-algebra spectrum.
For a map of spectra $f \colon b \ra bgl_1(R)$ as above we can ask whether
there is a morphism of commutative $R$-algebra spectra from $M(f)$ to $A$.
As $M(f)$ is defined as a (homotopy) pushout, we see that we get a condition
that says that we get maps from the ingredients of the derived smash product.
As we start with a map $\alpha$ from $R$ to $A$, we get an induced map
$$ gl_1(\alpha) \colon gl_1(R) \ra gl_1(A).$$
So what is missing is a map
$$(\Sigma^\infty_\mathbb{L})_+ \Omega_S^\infty p \ra A$$
that is compatible with the map
$(\Sigma^\infty_\mathbb{L})_+ \Omega_S^\infty gl_1(R) \ra A$.
With the help of the adjunction this means that we need a  map
$$ p \ra gl_1(A)$$
such that precomposing it with the map $gl_1(R) \ra p$ gives $gl_1(\alpha)$.
This argument can be turned into a proof for the following result:
\begin{thm} \, \cite[Theorem 4.6]{abghr} \, 
The derived mapping space of commutative $R$-algebras from $M(f)$ to $A$,
$\text{Map}_{\C_R}(M(f), A)$, is
weakly equivalent to the fiber in the map between derived mapping spaces
$$ \text{Map}_{\M_S}(p, gl_1(A)) \ra \text{Map}_{\M_S}(gl_1(R), gl_1(A))$$
at the basepoint $gl_1(\alpha)$ of $\text{Map}_{\M_S}(gl_1(R), gl_1(A))$.
\end{thm}

An important example is the question of the string orientation of 
the spectrum of topological modular forms, $\tmf$. For background on $\tmf$ 
and its variants see \cite{tmf}. In particular, \cite[Chapter 
10]{tmf} contains Andr\'e Henriques' notes of Mike Hopkins' lecture on the
string orientation.   
Let $BO\langle 8\rangle$ be the $7$-connected cover of $BO$ and
let $bo\langle 8\rangle$ be the associated spectrum with the canonical
map $f \colon bo\langle 8\rangle \ra bgl_1(S)$. So we are in the situation
where $R=S$ and we take $A = \tmf$. Ando, Hopkins and Rezk \cite{ahr}
establish the existence of an $E_\infty$-map
$$ M\text{String}=MO\langle 8\rangle \ra \tmf$$
by showing a fiber property as above.

An approach to orientations of the form $MU \ra E$ is
described in \cite{hopkins-lawson}: You start with an $E_\infty$ ring
spectrum $E$ and an ordinary complex
orientation of $E$ \cite[\S 6.1]{rezk}  and want to know whether you
can refine this to an $E_\infty$-map $MU \ra E$. Hopkins and Lawson
establish a filtration of $MU$ by $E_\infty$-Thom spectra 
$$ S \ra MX_1 \ra MX_2 \ra \ldots \ra MU $$ 
and for a given $E_\infty$-map $MX_n \ra E$ they identify the space of
extensions to an $E_\infty$-map $MX_{n+1} \ra E$ \cite[Theorem
1]{hopkins-lawson}.

\begin{rem}
In \cite{abghr} the authors present a different approach to Thom spectra and
questions about orientations that uses $\infty$-categorical
techniques. In certain cases it is unrealistic to hope for
$E_\infty$-maps out of Thom spectra, for instance if one doesn't know
that the target spectrum carries an $E_\infty$ structure. 
The space of $E_n$-maps out of Thom spectra is described  in
\cite[Theorem 4.2]{cm}, \cite[Corollary 3.18]{acb}. 
\end{rem}

\subsection{Thom spectra via $I$-spaces}
Let $I$ be the skeleton of the category of finite sets and injective maps. As
objects we choose the sets $\bfn = \{1, \ldots, n\}$ for $n \geq 0$ with the
convention that $\bf0$ denotes the empty set. A morphism
$f \in I(\bfn, \bfm)$ is an injective function from $\bfn$ to
$\bfm$. Hence $\bf0$ is an initial object of $I$ and the permutation group
$\Sigma_n$ is the group of automorphisms of $\bfn$ in $I$. The category $I$
is symmetric monoidal with respect to the disjoint union: $\bfn \sqcup \bfm
= \mathbf{n+m}$ with unit $\bf0$ and non-trivial symmetry $\mathbf{n+m}
\ra \mathbf{m+n}$ given by the shuffle permutation that moves the first $n$
elements to the positions $m+1, \ldots, m+n$.

The functor category of $I$-spaces, $\T^I$, \ie, the category of functors
$X \colon I
\ra \T$ together with natural transformations as morphisms,
inherits a symmetric monoidal structure from $I$ and $\T$ via the
Day convolution product. Explicitly, one gets:
\begin{defn}
Let $X, Y$ be two $I$-spaces. Their product $X \boxtimes Y$ is the $I$-space
given by
$$ (X \boxtimes Y)(\bfn) = \colim_{\mathbf{p} \sqcup \mathbf{q} \ra \bfn}
X(\mathbf{p}) \times Y(\mathbf{q}).$$
The unit $1_I$ is the discrete $I$-space $\bfn \mapsto 1_I(\bf0, \bfn)$.
\end{defn}
As $\bf0$ is initial, the unit $1_I$ is the terminal object in $\T^I$. 
Commutative monoids in $\T^I$ are called \emph{commutative $I$-space monoids}
in \cite{BSS} and their category is denoted by $C(\T^I)$. A general fact
about Day convolution products is that
commutative monoids correspond to lax symmetric monoidal functors.

For an $I$-space $X$ let's denote by $X_{hI}$ the Bousfield-Kan homotopy
colimit of $X$.

\begin{defn} \, \cite[Definition 2.2]{BSS} \, A map of
$I$-spaces $f\colon X \ra Y$ is \emph{an $I$-equivalence}, if the
induced map on homotopy colimits $f_{hI} \colon X_{hI} \ra Y_{hI}$ is a
weak homotopy equivalence in $\T$.
\end{defn}
With the corresponding $I$-model structure the category of $I$-spaces
is actually Quillen equivalent to the category of spaces \cite[Theorem
3.3]{sagave-sch}, but there is a \emph{positive
flat model structure} on $I$-spaces (see \cite[\S 2]{BSS}) that lifts to
a right-induced model structure on $C(\T^I)$ that makes it Quillen
equivalent to $E_\infty$-spaces.

Let $\Sp^\Sigma$ denote the category of symmetric spectra. There is a
canonical Quillen adjoint functor pair 
\begin{equation}  \label{eq:adjI}
\xymatrix@1{ {\T^I}  \ar@<0.5ex>[rr]^{\mathbb{S}^I} & & {\Sp^\Sigma}
\ar@<0.5ex>[ll]^{\Omega^I} }
\end{equation}
modelling the suspension spectrum functor and the underlying infinite loop
space functor with
$$ \mathbb{S}^IX(n) = \mathbb{S}^n \wedge X(\bfn), \quad \Omega^I(E)(\bfn)
= \Omega^nE_n$$
Where $\mathbb{S}^n$ is the $n$-fold smash product of the $1$-sphere with
$\Sigma_n$-action given by permutation of the smash factors.

Stable equivalences in symmetric spectra do in general not agree with stable 
homotopy equivalences, but there is a notion of  \emph{semistable} symmetric 
spectra, that has the feature that a map $f \colon E \ra F$ between 
two semistable symmetric spectra is a stable equivalence if and only if it is 
a stable homotopy equivalence. 
See \cite[\S 5.6]{HSS} for details and other characterizations. 

\begin{defn} For a 
commutative semistable symmetric ring spectrum $R$ the commutative 
$I$-space monoid of units, $GL_1^I(R)$, has as $GL_1^I(R)(\bfn)$ those 
components of the commutative $I$-space monoid $\Omega^I(R)(\bfn) = 
\Omega^nR_n$ that represent units in $\pi_0(R)$. 
\end{defn} 
The adjunction from \eqref{eq:adjI} gives a version of the assembly map from 
\eqref{eq:assembly} as $$ \mathbb{S}^I(GL_1^I(R)) \ra 
\mathbb{S}^I\Omega^I(R) \ra R.$$

For technical reasons one has to work with a cofibrant replacement of
$GL_1^I(R)$, $G \ra GL_1^I(R)$ in the positive flat model structure on
$C(\T^I)$. The construction of a Thom spectrum associated to a map $f \colon
X \ra BG$ is now similar to the approach in \cite{abghr}: One defines $BG$
and $EG$ via two sided-bar constructions and takes a suitable pushout:

\begin{defn} \cite[Definitions 2.10, 2.12, 3.6]{BSS}
\begin{itemize}
\item
Let $BG = B_\boxtimes(1_I, G, 1_I)$ and let $EG$ be defined via a
functorial factorization
$$\xymatrix@1{{B_\boxtimes(1_I, G, G) \phantom{b1}} \ar@{>->}[r]^(0.6)\sim & {EG}
\ar@{->>}[r] & {BG}}.$$
\item
For any $I$-space $X$ over $BG$ define $U(X)$ as the $I$-space with $G$-action
given by the pullback
$$\xymatrix{
{U(X)} \ar@{.>}[r] \ar@{.>}[d] & {X} \ar[d] \\
{EG} \ar[r] & {BG.}
}$$
Here, $X$ and $BG$ are considered as $I$-spaces with trivial $G$-action.
\item
Let $R$ be a semistable commutative symmetric ring spectrum that is
$S$-cofibrant.  
 
The \emph{Thom spectrum associated with a map of $I$-spaces $f\colon X
  \ra BG$} is  
\begin{equation} \label{eq:thomI}
M^I(f) = B_\boxtimes(\mathbb{S}^I(UX), \mathbb{S}^IG, R).
\end{equation}
\end{itemize}
\end{defn}
You should think of this two-sided bar construction as
$$ \mathbb{S}^I(UX) \boxtimes^L_{\mathbb{S}^IG} R$$
and then you have to admit that this looks very similar to
\eqref{eq:thomass}.
This Thom spectrum functor is homotopically meaningful (see
\cite[Proposition 3.8]{BSS}). Concerning multiplicative structures one gets
the following result.
\begin{prop} \, \cite[Proposition 3.10, Corollary 3.11]{BSS} \, 
The functor $M^I(-)$ is lax symmetric monoidal and if $\mathcal{D}$
is an operad in spaces, then it sends $\mathcal{D}$-algebras in $\T^I$
over $BG$ to $\mathcal{D}$-algebras in $R$-modules in symmetric spectra
over $M^IGL_1(R) = B_\boxtimes(\mathbb{S}^I(EG), \mathbb{S}^I(G), R)$.
\end{prop}

If you don't like diagram categories for some reason, then there is also
an \emph{$I$-spacification functor} \cite[\S 4.1]{BSS} that transforms a
map of topological spaces 
\begin{equation} \label{eq:bghi}
f \colon X \ra BG_{hI}
\end{equation} 
to a map of $I$-spaces
over $BG$, so you can associate a Thom spectrum to such a map as well. By
abuse of notation, we will still denote this Thom spectrum by $M^I(f)$.
This construction respects actions of operads augmented over the
Barratt-Eccles operad and hence it also provides an $E_{\infty}$
Thom spectrum functor.

An important question is: Given a ring spectrum $A$, can it be realized as a
Thom spectrum with respect to a loop map, \ie, in the setting of \cite{BSS}
is $A$ equivalent to $M^I(f)$ with
$f$ a loop map to $BG_{hI}$? A striking result is that one can identify
certain quotients as such Thom spectra!
\begin{thm} \label{thm:thomquotients} \, \cite[Theorem 5.6]{BSS} \,
Let $R$ be a commutative ring spectrum whose homotopy groups are concentrated
in even degrees and let $u_i \in \pi_{2i}(R)$ be arbitrary elements with
$1 \leq i \leq n-1$. Then the
iterative cofiber of the multiplication maps by the $u_i$'s, $R/(u_1,\ldots,
u_{n-1})$ can be realized as the Thom spectrum of a loop map from $SU(n)$ to
$BG_{hI}$. In particular,  $R/(u_1,\ldots,
u_{n-1})$ can be realized as an associative ring spectrum.
\end{thm}
An example of such a quotient is $R=ku \ra ku/u = H\Z$.
Note that there is no assumption on the regularity of the elements $u_i$
in the above statement. For periodic ring spectra the assumptions on the 
degree of the elements can be relaxed and the two-periodic version of 
Morava K-theory can be constructed as a Thom spectrum relative to $R =
E_n$, the $n$-th Morava $E$-theory or Lubin-Tate spectrum
\cite[Corollary 5.7]{BSS}. A related but 
different construction of quotients of 
Lubin-Tate spectra modelling versions of Morava K-theory is carried
out in \cite[\S 3]{hopkins-l}.

\subsection{Graded units}
There is one problem with the constructions of spaces and spectra of units
as above. As they are constructed from the underlying infinite loop space of
a spectrum and just take into account the units in $\pi_0$,  they
ignore 
graded units coming from periodicity elements in the homotopy groups of a
spectrum. So for instance, the Bott class $u \in \pi_2(\KU)$ is not
represented in $GL_1(\KU)$ or $GL_1^I(\KU)$.

There is a construction of \emph{graded units}. We'll sketch the construction
and mention two of its applications: graded Thom spectra and logarithmic ring
spectra.

\begin{defn} \, \cite[Definition 4.2]{sagave-sch} \, The category $J$ has
as objects pairs of objects of $I$. A morphism in $J((\bfn_1,\bfn_2),
(\bfm_1, \bfm_2))$ is a triple $(\alpha, \beta, \sigma)$ where $\alpha \in
I(\bfn_1, \bfm_1)$, $\beta \in I(\bfn_2, \bfm_2)$ and $\sigma$ is a
bijection
$$\sigma \colon \bfm_1 \setminus \alpha(\bfn_1) \ra
\bfm_2 \setminus \beta(\bfn_2).$$
For another morphism $(\gamma, \delta, \xi) \in
J((\bfm_1, \bfm_2), (\mathbf{l}_1, \mathbf{l}_2))$ the composition is the
morphism $(\gamma \circ \alpha, \delta \circ \beta, \tau(\xi, \sigma))$ where
$\tau(\xi, \sigma)$ is the permutation
$$ \tau(\xi, \sigma)(s) =
\begin{cases} \xi(s), & \text{ if } s \in
\mathbf{l}_1 \setminus \gamma(\bfm_1), \\
\delta(\sigma(t)), & \text{ if } s = \gamma(t) \in \gamma(\bfm_1
\setminus \alpha(\bfn_1)).
\end{cases} $$
\end{defn}
Note that $\mathbf{l}_1 \setminus \gamma(\alpha(\bfn_1))$ is the disjoint
union of $\mathbf{l}_1 \setminus \gamma(\bfm_1)$ and $\gamma(\bfm_1
\setminus \alpha(\bfn_1))$.

With these definitions $J$ is actually a category and it inherits a
symmetric monoidal structure
from $I$ via componentwise disjoint union \cite[Proposition 4.3]{sagave-sch}.
In particular, the category of $J$-spaces, $\T^J$, is symmetric monoidal
with the Day convolution product. Note, however, that the unit for the
monoidal structure $\boxtimes_J$ is $J((\bf0, \bf0), (-,-))$ and this is not
a constant functor but $J((\bf0, \bf0), (\bfn, \bfn))$ can be identified 
with the symmetric group $\Sigma_n$!

\begin{prop} \, \cite[4.4, 4.5]{sagave-sch} \, 
For every $J$-space $X$ the homotopy colimit, $X_{hJ}$, is a space over $QS^0$.
\end{prop}
\begin{proof}
It is not hard to see that $J$ is isomorphic to Quillen's category
$\Sigma^{-1}\Sigma$ \cite[Proposition 4.4]{sagave-sch} and its classifying
space is $QS^0$ by the Barratt-Priddy-Quillen result. Therefore $BJ$ is
$QS^0$. Every $J$-space has a map to the terminal $J$-space that is the
constant $J$-diagram on a point and this induces a map
$$ X_{hJ} \ra {\ast}_{hJ} = BJ \simeq QS^0.$$
\end{proof}
For any $I$-space $X$ we also get that $X_{hI}$ is a space over $BI$, but
as $I$ has an initial object this just gives a map to $BI \simeq \ast$, the
terminal object.

Let $C(\T^J)$ denote the category of commutative $J$-space monoids, \ie,
commutative monoids in $\T^J$. The following result is crucial:
\begin{thm} \, \cite[Theorem 4.11]{sagave-sch} \,
There is a model structure on $C(\T^J)$ such
that there is a Quillen
equivalence between $C(\T^J)$ and the category of $E_\infty$-spaces over
$BJ$.
\end{thm}
Here, the $E_\infty$-structure is parametrized by the Barratt-Eccles operad.

For a (commutative) $J$-space monoid, one can associate units:
\begin{defn} \, \cite[\S 4]{sagave-sch} \,
Let $A$ be a $J$-space monoid. Then let $A^\times$ be the $J$-space monoid
with $A^\times(\bfn_1, \bfn_2)$ being the union of those components of
$A(\bfn_1, \bfn_2)$ that represent units in $\pi_0(A_{hJ})$.
\end{defn}

So now one has to construct a functor from spectra to $J$-spaces that sees
all the homotopy groups, not just the ones in non-negative degrees:

\begin{defn} \, \cite[(4.5)]{sagave-sch} \,
\begin{itemize}
\item
Let $\Omega^J$ be the functor from symmetric spectra to $J$-spaces that takes
a symmetric spectrum $E$ and sends it to the $J$-space with
$$ \Omega^J(\bfn_1, \bfn_2) = \Omega^{n_2}E_{n_1}.$$
\item
If $R$ is a symmetric ring spectrum, then its \emph{$J$-space of units} is
$$ GL_1^J(R) = (\Omega^J(R))^\times.$$
\end{itemize}
\end{defn}
Sagave and Schlichtkrull show that this is homotopically meaningful and that
for a commutative symmetric ring spectrum $R$, the units $GL_1^J(R)$ are
actually in $C(\T^J)$ \cite[\S 4]{sagave-sch}. Most  importantly, the
inclusion $GL_1^J(R) \hookrightarrow \Omega^J(R)$ realizes the inclusion of
graded units $\pi_*(R)^\times$ into $\pi_*(R)$ for positively fibrant $R$.

Hence, for instance $GL_1^I(\KU)$ (and any other model of the 'usual' units)
only detects the units $\pm 1$ in $\pi_0(\KU)$ whereas $GL_1^J(\KU)$ also
detects the Bott class.

\begin{rem}
\begin{enumerate}
\item[]
\item
John Rognes developed the concept of logarithmic ring spectra and in
\cite{sagave} and \cite{rss} this concept is fully explored with the
help of graded units. The idea is that you want a spectrum that sits between
a commutative ring spectrum like $ku$ and its localization $\KU$, so you
remember the Bott class as the extra datum of a logarithmic structure. This
concept has its origin in algebraic geometry and is useful in stable
homotopy theory for instance for obtaining localization sequences in
topological Hochschild homology \cite{rss}.
\item
In \cite{ssthom} Sagave and Schlichtkrull use graded units adapted to the
setting of orthogonal spectra, $GL_1^W$,  to construct
\emph{graded Thom spectra}
associated to virtual  vector  bundles, \ie, associated to a map
$f \colon X \ra \Z \times BO$ in such a way that uses the
$E_\infty$-structure on $\Z \times BO$. They use this for orientation
theory and relate $GL_1^W$-orientations to logarithmic
structures. They provide an $E_{\infty}$-Thom isomorphism 
that allows to compute the homology of spectra appearing in
connection with logarithmic ring spectra \cite[\S \S \, 7,8]{ssthom}.  
\end{enumerate}
\end{rem}

\section{Constructing commutative ring spectra from bipermutative categories}
\label{sec:segal}

In section \ref{sec:thom} we saw that Thom spectra give rise to
commutative ring spectra. Algebraic K-theory is another
machine that takes a commutative ring (spectrum) $R$ and produces a
commutative ring spectrum $K(R)$. In this section we focus on a classical
construction that takes a
small bipermutative category $\R$ and turns it into a commutative ring
spectrum. This construction goes back to Segal \cite{catsandco}; its
multiplicative
properties were investigated by May
\cite{may-einfty,may-perm,may-mult,may-biperm}, Shimada-Shimakawa
\cite{shishi}, Woolfson \cite{woolfson} and 
Elmendorf-Mandell \cite{e-mandell}.

We sketch a simplified version of the construction, present some
important examples and
refer to  \cite{e-mandell} for a discussion of the multiplicative
properties.

\begin{defn}
A \emph{permutative category} $(\C, \oplus, 0, \tau)$ is a category
$\C$ together with an object $0$ of $\C$, a functor $\oplus \colon \C
\times \C \ra \C$  and a natural isomorphism $\tau_{C_1,C_2} \colon
C_1 \oplus C_2 \ra C_2 \oplus C_1$ for all objects $C_1, C_2$ of $\C$
such that
\begin{itemize}
\item
$\oplus$ is strictly associative, \ie, for all objects $C_1, C_2, C_3$
of $\C$
$$ C_1 \oplus (C_2 \oplus C_3) = (C_1 \oplus C_2) \oplus C_3.$$
\item
$0$ is a strict unit, \ie, for all objects $C$ of $\C$: $C \oplus 0 =
C = 0 \oplus C$.
\item
$\tau^2$ is the identity, \ie, for all objects $C_1, C_2$ of $\C$
the composite
$$\xymatrix@1{{C_1 \oplus C_2} \ar[r]^{\tau_{C_1,C_2}} &
  {C_2 \oplus C_1} \ar[r]^{\tau_{C_2,C_1}} & {C_1 \oplus C_2}}$$
is the identity on $C_1 \oplus C_2$.
\item
The diagrams
$$ \xymatrix{
{C \oplus 0} \ar^{\tau_{C,0}}[rr] \ar@{=}[dr] & &  {0 \oplus C} \ar@{=}[dl] \\
& {C, } &
}
\quad
\xymatrix{
{C_1 \oplus C_2 \oplus C_3} \ar[dr]_{\id_{C_1} \oplus \tau_{C_2, C_3}}
\ar[rr]^{\tau_{C_1 \oplus C_2, C_3}} & &
{C_3 \oplus C_1 \oplus C_2} \\
& {C_1 \oplus C_3 \oplus C_2} \ar[ru]_{\phantom{bla} \tau_{C_1, C_3} \oplus \id_{C_2}}&
}$$
commute for all objects $C, C_1, C_2, C_3$
of $\C$.
\end{itemize}
\end{defn}
We work with small permutative categories, \ie, we require that the
objects of $\C$ for a set (and not a proper class). We recall Segal's 
construction from \cite[\S 2]{catsandco}:
\begin{defn}
Let $\C$ be a small permutative category and let $X$ be a finite set
with basepoint $+ \in X$. Let $\C(X)$ be the category whose objects
are families $(C_S, \varrho_{S,T})$ where
\begin{itemize}
\item
$S \subset X$, $+ \notin S$,
\item
$S$ and $T$ are pairs of such subsets that are disjoint,
\item
the $C_S$ are objects of $\C$ and $\varrho_{S,T}$ is an isomorphism in
$\C$
$$ \varrho_{S,T}\colon C_S \oplus C_T \ra C_{S \cup T}.$$
\item
For $S = \varnothing$: $C_\varnothing = 0$ and $\varrho_{\varnothing,
  T} = \id_{C_T}$ for all $T$ and
\item
for pairwise disjoint $S, T, U$ that don't contain $+$ the following
diagrams commute:
$$ \xymatrix{
{C_S \oplus C_T} \ar[r]^{\varrho_{S,T}} \ar[d]_\tau & {C_{S \cup T}} \ar@{=}[d] \\
{C_T \oplus C_S} \ar[r]^{\varrho_{T,S}} & {C_{T \cup S},}}
\qquad
\xymatrix{
{C_S \oplus C_T \oplus C_U} \ar[rr]^{\varrho_{S,T} \oplus \id_{C_U}}
\ar[d]_{\id_{C_S} \oplus \varrho_{T,U}}  & & {C_{S \cup T} \oplus C_U}
\ar[d]^{\varrho_{S \cup T, U}} \\
{C_S \oplus C_{T \cup U}} \ar[rr]^{\varrho_{S, T \cup U}} & & {C_{S
    \cup T \cup U.}}
}$$
\end{itemize}
Morphisms $\alpha  \colon (C_S, \varrho_{S,T}) \ra (C'_S,
\varrho'_{S,T})$ consist of a family of morphisms $\alpha_S \in
\C(C_S, C'_S)$ for all $S \subset X$ with $+ \notin S$ such that
$\alpha_\varnothing = \id_0$ and for all $S, T \in X$ with $+ \notin
S, T$ and $S \cap T = \varnothing$ the diagram
$$\xymatrix{
{C_S \oplus C_T} \ar[r]^{\varrho_{S,T}} \ar[d]_{\alpha_S \oplus
  \alpha_T}& {C_{S \cup T}} \ar[d]^{\alpha_{S \cup T}} \\
{C'_S \oplus C'_T} \ar[r]^{\varrho'_{S,T}} & {C'_{S \cup T}}
}$$
commutes.
\end{defn}
Thus up to isomorphism, every object $C_S$ for $S=\{x_1,\ldots,x_n\}$
can be decomposed as
$$ C_S \cong C_{\{x_1\}} \oplus \ldots \oplus C_{\{x_n\}}$$
by an iterated application of the isomorphisms $\varrho$, but these
isomorphisms are part of the data. Segal shows \cite[Corollary
2.2]{catsandco} that this construction gives rise to a so-called
$\Gamma$-space (see \cite[Definition 1.2]{catsandco} for a
definition) that sends a finite pointed set $X$ to the 
classifying space of $\C(X)$. Every
$\Gamma$-space gives rise to a spectrum, and we
denote the spectrum associated to $\C$ by $H\C$.

\begin{rem}
Segal's construction actually works for symmetric monoidal categories and
it produces a spectrum whose associated infinite loop space is the group
completion of the
classifying space of the category $\C$, $B\C$, and the latter is
the geometric realization of the nerve of $\C$.
\end{rem}
\begin{defn} A \emph{bipermutative category}
$\R$ is a category with two permutative
category structures:
$(\R, \oplus, 0_\R, \tau_\oplus)$ and $(\R, \otimes, 1_\R, \tau_\otimes)$
that are compatible in the following sense:
\begin{enumerate}
\item
$$ 0_\R \otimes C = 0_\R = C \otimes 0_\R$$
for all objects $C$ of $\R$.
\item
For all objects $A, B, C$ we have an equality between $(A \oplus B) \otimes C$
and $A \otimes C \oplus B \otimes C$ and the diagram
$$ \xymatrix{
{(A \oplus B) \otimes C} \ar@{=}[r] \ar[d]_{\tau_\oplus \otimes \id} &
{A \otimes C \oplus B \otimes C} \ar[d]^{\tau_\oplus}\\
{(B \oplus A) \otimes C} \ar@{=}[r] & {B \otimes C \oplus A \otimes C}
}$$
commutes.
\item
We define the distributivity isomorphism
$d_\ell \colon A \otimes (B \oplus C) \ra A \otimes B \oplus A \otimes C$ for
all $A, B, C$ in $\R$ via
the following diagrem
$$ \xymatrix{
{A \otimes (B \oplus C)} \ar[r]^{\tau_\otimes} \ar@{.>}[d]_{d_\ell} &
{(B \oplus C) \otimes A} \ar@{=}[d] \\
{A \otimes B \oplus A \otimes C} & {B \otimes A \oplus C \otimes A}
\ar[l]^{\tau_\otimes \oplus \tau_\otimes}
}$$
then the diagram
$$\xymatrix{
{(A \oplus B) \otimes (C \oplus D)} \ar@{=}[dd] \ar[dr]^{d_\ell}& {} \\
{} & {(A \oplus B) \otimes C \oplus (A \oplus B) \otimes D} \ar@{=}[dd]\\
{A \otimes (C \oplus D) \oplus B \otimes (C \oplus D)}
\ar[dd]_{d_\ell \oplus d_\ell}& {} \\
{} & {A \otimes C \oplus B \otimes C \oplus A \otimes D \oplus B \otimes D} \\
{A \otimes C \oplus A \otimes D \oplus B \otimes C \oplus B \otimes D}
\ar[ur]^{\id \oplus \tau_\oplus \oplus \id}& {}
}$$
commutes.
\end{enumerate}
\end{defn}
This definition is taken from \cite[Definition VI.3.3, p.~154]{may-einfty}.
The definition in \cite{e-mandell} is less strict, but bipermutative
categories in the above sense are also bipermutative in the sense of
\cite[Definition 3.6]{e-mandell}.
We refer to Elmendorf and Mandell for a proof
that for a bipermutative category $\R$, one
actually obtains a commutative ring spectrum $H\R$:
\begin{thm} \, \cite[Corollary 3.9]{e-mandell} \, 
If $\R$ is a bipermutative category, then $H\R$ is equivalent to a strictly
commutative symmetric ring spectrum.
\end{thm}
Segal's construction enables us to find small and explicit models for
certain connective commutative ring spectra. Famous examples of bipermutative
categories and their associated commutative ring spectra
are the following:
\begin{enumerate}
\item
If $R$ is a commutative discrete ring, then the category $\R_R$ which has the
elements of $R$ as
objects and only identity morphisms is a bipermutative category with
the addition in the ring being $\oplus$ and the multiplication being
$\otimes$. The associated spectrum, $H\R_R$ is the Eilenberg-Mac\,Lane
spectrum of the ring $R$, $\HR$.
\item
Let $\mathcal{E}$ denote the bipermutative category of finite sets
whose objects are the finite sets $\mathbf{n} = \{1,\ldots,n\}$ for
$n \in \mathbb{N}_0$. By convention $\mathbf{0}$ is the empty set.
The morphisms in $\mathcal{E}$ are
$$ \mathcal{E}(\mathbf{n},\mathbf{m}) = \left\{
\begin{array}{cc}
\varnothing, & n \neq m, \\
\Sigma_n, & n = m.
\end{array}
\right.$$ For the full structure see \cite[VI, Example
5.1]{may-einfty}. Here, $H\mathcal{E}$ is the sphere spectrum, $S$.
\item
The bipermutative category of complex vector spaces,
$\mathcal{V}_\mathbb{C}$, with objects the natural numbers with zero
and morphisms
$$ \mathcal{V}_\mathbb{C}(n,m) = \left\{\begin{array}{cc}
\varnothing, & n \neq m, \\
U(n), &  n=m
\end{array} \right. $$
is bipermutative. On objects we set $n \oplus m = n+m$ and $n \otimes
m = nm$ and on morphism we use the block sum and the tensor product of
matrices. The associated spectrum is $H\mathcal{V}_\mathbb{C} = ku$,
the connective version of topological complex K-theory. Its real
analog, $\mathcal{V}_\mathbb{R}$, gives a model for connective
topological real K-theory, $ko$. You can also work with the general
linear group instead of the unitary or orthogonal group.
\item
If $R$ is a discrete commutative ring, then we define the category
$F_R$ as the one with objects $\mathbb{N}_0$ again. As morphisms we have
$$ F_R(n,m) = \begin{cases}
\varnothing, & n \neq m,\\
GL_n(R), &  n=m.
\end{cases}$$
This category is often called \emph{the small category of free $R$-modules}.
Its spectrum is the \emph{free algebraic K-theory of $R$}, $K^f(R)$. Its
homotopy groups agree with the algebraic K-groups of $R$ from degree $1$ on.
\end{enumerate}

\section{From topological Hochschild to topological Andr\'e-Quillen homology}
\label{sec:thhetc}
For rings and algebras Hochschild homology contains a lot of information.
For commutative rings and algebras Andr\'e-Quillen homology is the
adequate tool. There are spectrum level versions of these homology theories:
topological Hochschild homology, $\THH$, and topological Andr\'e-Quillen
homology, $\TAQ$.

We can determine classes in the algebraic K-theory of a ring spectrum
using the trace to topological
Hochschild homology or to topological cyclic homology:
\begin{equation} \label{eq:trace} tr \colon K(R) \ra \THH(R). \end{equation}
For instance the trace from $K(\mathbb{Z})$ to $\THH(\mathbb{Z})$
detects important classes. B\"okstedt, Madsen and Rognes
\cite{boemad,rognestrace}
show for instance that
$$\text{tr}\colon K_{2p-1}(\Z) \ra \THH_{2p-1}(\Z) \cong \Z/p\Z  $$
is surjective for all primes $p$.

We give a construction of topological Hochschild homology and more generally,
for commutative ring spectra $R$  we define $X \otimes R$ for $X$ a finite
pointed simplicial set. We give some examples of calculations of such
$X$-homology groups of $R$ and tell you about topological Hochschild
cohomology as a derived center of an algebra spectrum. We define topological
Andr\'e-Quillen homology
and we will see applications to Postnikov towers for commutative ring
spectra later in \ref{sec:obs}.

\subsection{$\THH$ and friends}
Let $X$ be a finite pointed simplicial set and let $R$ be a cofibrant
commutative ring spectrum.
\begin{defn} \label{def:xtensorr}
We denote by $X \otimes R$ the simplicial spectrum with
$$ (X \otimes R)_n = \bigwedge_{x \in X_n} R.$$
By slight abuse of notation we will use the same symbol for the geometric
realization of $X \otimes R$.
\end{defn}
\begin{rem} \label{rem:augmentationTHH}
\begin{itemize}
\item[]
\item
As the smash product is the 
coproduct in $\C_S$, the simplicial structure maps of $X \otimes R$ are
induced from the ones on $X$.
\item
As $X$ is pointed, $X \otimes R$ comes with maps
$$ R \ra X \otimes R \ra R$$
whose composition is the identity.
\item
The commutative multiplication on $R$ induces a commutative multiplication on
$X \otimes R$, hence $X \otimes R$ is an augmented commutative (simplicial)
$R$-algebra spectrum.
\item
One could also work with the fact that the spectra of \cite{EKMM} are
tensored over topological spaces or similarly, that symmetric spectra
\cite{HSS} in
topological spaces are enriched over simplicial sets and over
topological spaces. This gives an equivalent situation. It is for
instance shown in \cite[Corollary VII.3.4]{EKMM} that $|X \otimes A|
\simeq |X| \otimes A$ for simplicial spaces $X$ and commutative
$R$-algebra spectra $A$.
\item
The above definition can be extended to tensoring with an arbitrary pointed
simplicial sets by
expressing such a simplicial set as the colimit of its finite pointed
simplicial subcomplexes.
\end{itemize}
\end{rem}
There are many important special cases of this construction.
\begin{defn}
\begin{enumerate}
\item[]
\item
For the simplicial $1$-sphere $X = \mathbb{S}^1$ the commutative $R$-algebra
spectrum  $\mathbb{S}^1 \otimes R$ is the \emph{topological Hochschild
homology of $R$} and is denoted by $\THH(R)$.
\item
More generally, for an $n$-sphere, we denote by $\THH^{[n]}(R)$ the
spectrum $\mathbb{S}^n \otimes R$ and this is called \emph{topological
Hochschild homology of order $n$}.
\item
If $\mathbb{T}^n$ denotes the torus $(\mathbb{S}^1)^n$, then
$\mathbb{T}^n \otimes R$ is the \emph{$n$-torus homology of $R$}.
\end{enumerate}
\end{defn}
For the small model of the simplicial $1$-sphere with just one non-degenerate
$0$ and $1$ simplex we have $(\mathbb{S}^1)_n = \{0,1,\ldots, n\}$ and
the simplicial spectrum $\mathbb{S}^1 \otimes R$ is
precisely the cyclic bar construction on $R$:
$$ \xymatrix@1{{R \, } \ar[r] & {\, R \wedge R \, }
\ar@<0.5ex>[l] \ar@<-0.5ex>[l]
\ar@<0.5ex>[r] \ar@<-0.5ex>[r] &
{\, R \wedge R \wedge R \, } \ar@<-1ex>[l] \ar[l] \ar@<1ex>[l]
\ar@<1ex>[r] \ar[r] \ar@<-1ex>[r]
& { \, \ldots}  \ar@<1.5ex>[l] \ar@<0.5ex>[l] \ar@<-0.5ex>[l] \ar@<-1.5ex>[l] } $$
where the degeneracy map $s_i \colon R^{n+1} \ra R^{n+2}$ inserts the unit
map $\eta \colon S \ra R$ after the $i$th factor of $R$  and
the face maps $d_i \colon  R^{n+1} \ra R^n$ for $0 \leq i < n$ are given by
the multiplication in $R$ of the $i$th and $(i+1)$st smash factor. The last
face map $d_n$ cyclically permutes the smash factors to bring the last one to
the front and then it multiplies the former factors with number $n$ and $0$.

As for Hochschild homology you should think about this as a genuine cyclic
object:

\begin{center}
\setlength{\unitlength}{1cm}
\begin{picture}(5,3)
\put(0,2){$R$}
\put(1,2.3){$R$}
\put(0,1){$R$}
\put(0.5,2.2){$\wedge$}
\put(1.5,2.3){$\wedge$}
\put(0,1.5){$\wedge$}
\put(0.5,0.7){$\wedge$}
\put(1,0.6){$\cdot$}
\put(1.2,0.55){$\cdot$}
\put(1.4,0.6){$\cdot$}
\put(2,2.1){$\cdot$}
\put(2.1,1.9){$\cdot$}
\put(2.1,1.7){$\cdot$}
\end{picture}
\end{center}

The original definition of $\THH$ is due to Marcel B\"okstedt
\cite{boe-thh}.  
McClure, Schw\"anzl and Vogt \cite{MSV} show that for an
$E_\infty$-ring spectrum $R$, $\THH(R)$, is equivalent to tensoring $R$
with the topological $1$-sphere. Kuhn systematically studies constructions
like $X \otimes R$ in a reduced setting \cite{kuhn} for pointed spaces $X$.
So the above definition is an unreduced variant of this that uses
simplicial sets instead of topological spaces.

\begin{lem} \label{lem:prod}
Let $X$ and $Y$ be finite simplicial pointed sets. Then
$$ (X \times Y) \otimes R \simeq X \otimes (Y \otimes R).$$
\end{lem}
\begin{proof}
Observe that
$$ ((X \times Y) \otimes R)_n = \bigwedge_{(x,y) \in X_n \times Y_n} R \cong
\bigwedge_{x \in X_n}(\bigwedge_{y \in Y_n} R)$$
and this is the diagonal of the bisimplicial spectrum
$$ ([m],[\ell]) \mapsto (X \otimes ((Y \otimes R)_\ell))_m$$
in degree $n$.
\end{proof}
One of the important features of $\THH(R)$ is that it receives a trace map from
algebraic K-theory (see \eqref{eq:trace}) that we can now write as
$$ tr \colon K(R) \ra \mathbb{S}^1 \otimes R. $$
Taking higher dimensional tori gives targets
for iterated trace maps. Algebraic K-theory of a commutative ring
spectrum is again a commutative ring spectrum and the trace map
is a map of commutative ring spectra, hence one can iterate the
process of forming K-theory and traces. If we denote by $K^n(R)$ the
$n$-fold iteration, 
then -- as we have the product formula from Lemma \ref{lem:prod} --
we get an iterated trace to $\mathbb{T}^n \otimes R$. Explicitly, for $n=2$
this is
$$ K(K(R))  \ra \mathbb{S}^1 \otimes (\mathbb{S}^1 \otimes R) \simeq
(\mathbb{S}^1 \times \mathbb{S}^1) \otimes R = \mathbb{T}^2 \otimes R.$$

There are variants of Definition \ref{def:xtensorr}: As we work with
pointed simplicial sets, we can glue an $R$-module to the base point and use
the $R$-module structure for the face maps. A second variant is to
work relative to some commutative ring spectrum $R$: 
In \ref{def:xtensorr} the smash products were
over the sphere spectrum, but if we work with a commutative
$R$-algebra spectrum $A$,
then we can take smash products over $R$ instead of $S$. Recall that $\wedge_R$
is the coproduct in the category of commutative $R$-algebra spectra, $\C_R$.

\begin{defn}
Let $R$ be a cofibrant commutative ring spectrum, $A$ a cofibrant commutative
$R$-algebra spectrum, $M$ an $A$-module spectrum over $R$ and let $X$ be a
finite pointed simplicial set. We denote by $\call^R_X(A;M)$ the
simplicial spectrum with 
$$ \call^R_X(A;M)_n = M \wedge_R {\bigwedge_{x\in X_n \setminus
    *}}\raisebox{-5pt}{}_{\!\!\!\!\!\!R} A. $$ 
We call $\call^R_X(A;M)$ the \emph{Loday construction of $A$ over $R$ with
coefficients in $M$}.
\end{defn}
As $M$ is just an $A$-module spectrum, the resulting simplicial spectrum and
also its realization carries an $A$-module structure over $R$, but no
multiplicative structure in general. However, if we place a commutative
$A$-algebra $C$ at the basepoint, then the resulting spectrum is an augmented
commutative $C$-algebra spectrum.

We will see later in Section \ref{sec:etale} that for instance
$$ \THH^R(A) := \call^R_{\mathbb{S}^1}(A)$$
measures properties of $A$ as a commutative $R$-algebra spectrum. The
case of $X = \mathbb{S}^0$ gives 
$$  \call^R_{\mathbb{S}^0}(A) = A \wedge_R A$$
so there is a K\"unneth spectral sequence \cite[IV.4.1]{EKMM} for calculating
its homotopy groups.

An important example of a Loday construction is Pirashvili's
construction of \emph{higher order Hochschild homology}.
He works with discrete
commutative $k$-algebras $A$ and  $A$-modules $M$  and defines
$\HH_X^k(A; M)$ \cite[\S 5.1]{pira}. For $X = \mathbb{S}^n$ this is his
notion of higher order Hochschild homology (in his notation $H^{[n]}(A; M)$).
In our setting this corresponds to $\call^{H\!k}_X(\HA; H\!M)$
if $A$ is flat over $k$.

\subsection{Examples}
\begin{enumerate}
\item
A classical example of a $\THH$-calculation is the one of $\HZ$ and $H\F_p$
by Marcel B\"okstedt (\cite{Boe}, see \cite[Chapter 13]{loday} and the
references for published accounts of these results):
\begin{prop}
$$ \THH_*(H\F_p) \cong \F_p[\mu], \quad |\mu| = 2.$$
$$ \THH_i(\HZ) \cong \begin{cases}\Z,  & i=0, \\
\Z/j\Z, & i = 2j-1, \\
0, & \text{ otherwise. } \end{cases}$$
\end{prop}

A crucial ingredient for these and for many other calculations of $\THH$ is 
B\"okstedt's spectral sequence: If $R$ is a commutative ring spectrum and if 
$E_*$ is a homotopy commutative ring spectrum such that $E_*(R)$ is flat over 
$E_*$ then there is a multiplicative spectral 
sequence 
$$ E^2_{p,q} = \HH^{E_*}_{p,q}(E_*(R)) \Rightarrow E_{p+q}\THH(R). $$ 
Here $\HH_{p,q}$ denotes Hochschild homology in homological degree $p$ and 
internal degree $q$ (\cite{Boe}, \cite[Theorem IV.1.9]{EKMM}). 

\item
If we apply $\THH$ to Eilenberg-Mac\,Lane spectra of number rings,
Lindenstrauss and Madsen show that $\THH$ detects arithmetic properties:
\begin{prop} \, \cite[Theorem 1.1]{lm} \, 
Let $K$ be a number field and let $\mathcal{O}_K$ be its ring of integers.
Then
$$ \THH_n(H\!\mathcal{O}_K) = \begin{cases}
\mathcal{O}_K, & n=0, \\
\mathcal{D}_{\mathcal{O}_K}^{-1}/\ell \mathcal{O}_K, & n = 2\ell-1, \\
0, & \text{otherwise.}
\end{cases} $$
\end{prop}
Here, $\mathcal{D}_{\mathcal{O}_K}^{-1}$ is the inverse different. 
This is the set of all $x \in K$ such that the trace $tr(xy)$ is an 
integer for all $y \in \mathcal{O}_K$. 
The inverse different detects ramified primes and
their ramification type. 

We calculate
higher order $\THH$ of number rings with reduced coefficients in
\cite[Theorem 4.3]{dlr}. 
\item
For a suspension spectrum on a based (Moore) loop space,
$\Sigma^\infty_+\Omega_M X$,
the cyclic bar construction reduces to the suspension spectrum of the
cyclic bar construction on $\Omega_M X$ and  Goodwillie \cite[Proof of
Theorem V.1.1]{goodwillie} identifies the latter with the free loop
space on $X$, $LX$. Hence one obtains
$$ \THH(\Sigma^\infty_+\Omega_M X) \simeq \Sigma^\infty_+ LX.$$
\item
Let $R$ be a ring spectrum and let $\Pi$ be a pointed monoid.
Hesselholt and Madsen show that $\THH(R[\Pi])$ splits as
$$ \THH(R[\Pi]) \simeq \THH(R) \wedge |N^{cy}\Pi|$$
where $|N^{cy}\Pi|$ denotes the cyclic nerve of $\Pi$
\cite[Theorem 7.1]{hesselholt-m}.
\item
As a sample calculation for second order $\THH$ we get \cite[Theorem 2.1]{dlr2}:
\begin{equation} \label{eq:thh2plocals}
\THH^{[2]}_*(H\Z_{(p)}) \cong
\Z_{(p)}[x_1, x_2, \ldots]/(p^nx_n, x_n^p-px_{n+1}, n \geq 1)
\end{equation}
with $|x_1|=2p$.
\item
At an odd prime $\KU_{\!\!(p)}$ splits as
$$ \KU_{\!\!(p)} \simeq \bigvee_{i=0}^{p-2} \Sigma^{2i} L.$$
Here, $L$ is the Adams summand of $\KU_{\!\!(p)}$ 
with $\pi_*(L) \cong \Z_{(p)}[v_1^{\pm 1}]$ and $|v_1|=2p-2$.   
For consistency we set $L = \KU_{\!\!(2)}$ at the prime $2$.  
We denote by $\ku$, $\ell$ and $\ko$ the connective covers of
$\KU$, $L$ and $\KO$. 

McClure and Staffeldt determine the mod $p$-homotopy of $\THH(\ell)$ at
odd primes
\cite{mccs} and they show that $\THH(L)_p \simeq L_p \vee (\Sigma L_p)_\Q$
\cite[Corollary 7.2, Theorem 8.1]{mccs}.

Ausoni \cite{ausoni} determines the mod $p$ and mod $v_1$ homotopy of
$\THH(ku)$ as an input for his work on  $K(ku)$.

Angeltveit, Hill and Lawson show \cite[Theorem 2.6]{ahl} that for all primes, 
as an $\ell_*$-module 
$$ \THH_*(\ell) \cong \ell_* \oplus \Sigma^{2p-1}F \oplus T$$
where $F$ is a torsionfree summand and $T$ is an infinite direct sum of
torsion modules concentrated in even degrees. They describe $F$
explicitly using a rational calculation. Determining the torsion is way more
involved \cite[Theorem 2.8]{ahl}. The calculation of $\THH_*(\ell)$  
uses the method of \emph{duelling Bockstein spectral 
sequences} for the Bockstein spectral sequences associated to
$$ \xymatrix{
{\ell} \ar[d] \ar[r] & {\ell/p} \ar[d] \\
{\ell/v_1 = H\!\Z_{(p)}} \ar[r] & {H\!\F_p=\ell/(p,v_1).}
}$$

They describe the $2$-local homotopy groups of $\THH(ko)$ \cite[\S 7]{ahl} by 
first determining $\THH_*(ko; ku)$ and then using the Bockstein spectral 
sequence associated to the cofiber sequence $\Sigma ko \ra ko \ra ku$. 

Again, things are way easier for the periodic versions
\cite[Proposition 7.13]{ausoni}, \cite[Corollary 7.9]{ahl}:
$$ \THH(\KO) \simeq \KO \vee \Sigma \KO_\Q, \quad \THH(KU) \simeq KU \vee
\Sigma \KU_\Q. $$

\item
John Greenlees uses a generalization of the concept of Gorenstein maps of 
commutative rings to the spectral world in order to determine Gorenstein 
descent properties for cofiber sequences of connective commutative ring spectra 
\cite[Theorem 7.4]{greenlees-go}. 
\end{enumerate}

\subsection{Topological Hochschild homology of Thom spectra}
We start with a general statement about $X \otimes M^I(f)$ if $M^I(f)$ is a Thom
spectrum associated to an $E_\infty$-map to $BG_{hI}$ with $BG_{hI}$ 
as in \eqref{eq:bghi} with $R=S$, hence $G$ is a cofibrant replacement
of $GL_1^I(S)$. 
\begin{thm} \, \cite[Theorem 1.1]{schlichtkrull} \, 
For any pointed simplicial set $X$ and any map of grouplike $E_\infty$-spaces
$f\colon A \ra BG_{hI}$ there is an equivalence of $E_\infty$-ring spectra
$$ X \otimes M^I(f) \simeq  M^I(f) \wedge \Omega^\infty (a \wedge |X|)_+$$
where $a$ is the spectrum associated to $A$ with $\Omega^\infty a = A$.
\end{thm}

This result generalizes \cite{bcs}, where the case of $X=\mathbb{S}^1$ is
covered. In general, for $X=\mathbb{S}^n$ the above result determines the
higher order topological Hochschild homology of $M^I(f)$
\cite[(1.6)]{schlichtkrull} as 
$$ \THH^{[n]}(M^I(f)) \simeq M^I(f) \wedge B^nA_+.$$

As an example consider the canonical map $f\colon BU \ra BG_{hI}$. Then one
obtains
$$ X \otimes MU \simeq MU \wedge \Omega^\infty (bu \wedge |X|)$$
and
$$ \THH^{[n]}(MU) \simeq MU \wedge \Omega^\infty (bu \wedge \mathbb{S}^n)
\simeq MU \wedge B^nBU_+.$$

There is also a statement about $\THH$ of Thom spectra associated to
single loop maps in \cite[Theorem 1]{bcs}.
We state the relative version of this, so $G$ is a cofibrant
replacement of $GL_1^I(R)$. 
\begin{thm} \, \cite[Theorem 6.6]{BSS} \, 
Assume that $R$ is a commutative symmetric ring spectrum that is semistable
and $S$-cofibrant.
Let $M^I(f)$ be a Thom spectrum associated to a map $f\colon M \ra BG_{hI}$
of topological monoids, where $M$ is grouplike and well-pointed. Then
$$ \THH^R(M^I(f)) \simeq M^I(L^\eta(B(f))).$$
\end{thm}
Here, $M^I(L^\eta(B(f)))$ is the Thom spectrum associated to the map
$$ \xymatrix{
{L(B(M))} \ar[rr]^(0.4){L(B(f))} \ar@{.>}[d]_{L^\eta(B(f))}  & &
{LBBG_{hI} \simeq  BG_{hI} \times BBG_{hI}} \ar[d]^{\id \times \eta} \\
{BG_{hI}} & & { BG_{hI} \times BG_{hI} } \ar[ll]_\mu
}$$
Note that $BBG_{hI}$ is an $H$-group, so we can split the free loop
space $LBBG_{hI}$ into the base space and the based loops
$$ LBBG_{hI} \simeq BBG_{hI} \times \Omega BBG_{hI}$$
and the second factor is equivalent to $BG_{hI}$. As usual, $\eta$
denotes the Hopf map $\eta\colon \mathbb{S}^3 \ra \mathbb{S}^2$ and it
induces a map $\eta\colon BBG_{hI} \ra BG_{hI}$ as above via
$$ BBG_{hI} \simeq \Omega^2B^4G_{hI} \ra \Omega^3B^4G_{hI} \simeq BG_{hI}$$
by reducing the loop coordinates by precomposition.

For quotient spectra, this results gives a new way of calculating
$\THH^R(R/I)$. For related results see \cite{angeltveit} and in the case
where $R/I$ happens to be commutative see \cite[\S 7]{dlr2}.

A second example comes from viewing $H\Z_{(p)}$ as a Thom spectrum
associated as a $2$-fold loop map $\Omega^2(\mathbb{S}^3\langle 3\rangle)
\ra BG_{hI}$ which allows for a determination of $\THH(H\Z_{(p)})$
as $H\Z_{(p)}\wedge \Omega(\mathbb{S}^3\langle 3\rangle)_+$
\cite[Theorem 3.8]{bcs} and an additive equivalence
$$\THH^{[2]}(H\Z_{(p)})
\simeq H\Z_{(p)}\wedge \mathbb{S}^3\langle 3\rangle_+.$$
This gives a geometric interpretation of \eqref{eq:thh2plocals}, but
without an identification of the multiplicative structure. 
See also \cite[\S 4]{klang}, where Klang presents related results, using
the framework of factorization homology.

\subsection{Topological Hochschild cohomology as a derived center}
In the discrete case, \ie, for a commutative ring $k$ and a $k$-algebra $A$
one can describe the center of $A$,
$$Z(A) = \{b \in A, ab =ba \text{ for all } a \in A\}$$
as the set of $A$-bimodule maps from $A$ to $A$. If $f$ is such a map,
$f \colon A \ra A$ with $f(cad) = cf(a)d$ for all $a,c,d \in A$, then $f$ is
determined by $f(1) =: b$ and this $b$ satisfies
$$ ab = af(1) = f(a\cdot 1) = f(a) = f(1\cdot a) = f(1)a = ba$$
so the set of such morphisms gives rise to an element in the center and vice
versa, for any $b \in Z(A)$ we get such an $f$ by setting $f(1) =b$.

Hochschild cohomology of $A$ over $k$ can be described as
$$ \HH^*_k(A) = \Ext^*_{A \otimes_k A^{o}}(A, A)$$
if $A$ is $k$-projective. Hence $\HH^0_k(A) = Z(A)$ and the Hochschild
cohomology of $A$ is the \emph{derived center of $A$}. Hochschild cohomology
has a graded commutative algebra structure via a cup product,
 but the solved Deligne conjecture \cite{mcc-smith} says that the Hochschild
cochain complex is in general not a differential graded commutative
algebra, but that it has an $E_2$-algebra structure.

For ring spectra there is no homotopically meaningful definition of a center:
requiring equality translates to an equalizer diagram and this wouldn't be
homotopy invariant. For a commutative ring spectrum $R$ and an $R$-algebra
spectrum $A$ this equalizer corresponds precisely to taking not
just $R$-module endomorphisms but $A$-bimodule endomorphisms.
So a homotopy invariant version is as follows.
\begin{defn} \label{def:derivedcenter}
For a commutative ring spectrum $R$ and an $R$-algebra spectrum $A$, the
\emph{topological Hochschild cohomology groups of $A$ over $R$} are
$$\THH_R^*(A) = \pi_*\Ext_{A \wedge_R A^{o}}(A, A)$$
and the \emph{derived center of $A$ over $R$} is
$$ \THH_R(A) = \Ext_{A \wedge_R A^{o}}(A, A).$$
Here, $\Ext_{A \wedge_R A^{o}}(A, A)$ denotes the derived endomorphism
spectrum of $A$ as an $A$-bimodule \cite[IV \S 1]{EKMM}.
\end{defn}
McClure and Smith's proof of the Deligne conjecture also provides a  spectrum
version for topological Hochschild cohomology, giving the derived
center an $E_2$-structure:  
\begin{thm} \, \cite{mcc-smith} \, 
If $A$ is an associative $R$-algebra spectrum, then $\THH_R(A)$ is an
$E_2$-ring spectrum.
\end{thm}

An important example of a calculation of such a derived center is
Angeltveit's calculation of $\THH_{E_n}(K_n)$. Here $E_n$ denotes Morava
$E$-theory with
$$ \pi_*(E_n) \cong W(\F_q)[[u_1,\ldots,u_{n-1}]][u^{\pm 1}]$$
where the $u_i$ are deformation parameters for the height $n$ Honda formal
group law with $|u_i|=0$ and $u$ is a periodicity element with $|u|=2$. The
sequence of elements $(p,u_1,\ldots,u_{n-1})$ is a regular sequence and
$K_n$ is the $2$-periodic version of Morava $K$-theory
$$ K_n = E_n/(p,u_1,\ldots,u_{n-1}), \quad (K_n)_* = \F_q[u^{\pm 1}]. $$

Angeltveit shows that the derived center of $K_n$ over $E_n$ depends on
the chosen $A_\infty$-algebra structure of $K_n$ over $E_n$:
\begin{thm} \, \cite[Theorems 5.21, 5.22]{angeltveit} \,
\begin{enumerate}
\item
For any prime $p$ and any $n \geq 1$ there is an  $A_\infty$-structure on
$K_n$ such that $\THH_{E_n}(K_n) \simeq E_n$. 
\item
For $n=1$ and any $d$ with $1 \leq d < p-1$ and any $a$ with $1 \leq a \leq 
p-1$ there is an $A_\infty$-structure on $K_1$ with 
$$ \THH^*_{E_1}(K_1) \cong \pi_*(E_1)[[q]]/(p+a(uq)^d).$$ 
\end{enumerate}
\end{thm}
Here, the first structure in (a) occurs as the one coming from the
\emph{least commutative $A_\infty$ structure} on $K_n$
(see \cite[Theorem 5.8]{angeltveit} for a precise statement). The case
$n=1, p=2$ of (a) is due to Baker and Lazarev \cite[Proof of Theorem 3.1]{bl} 
who show that at the prime $2$
$$ \THH_{KU_2}(K(1)) \simeq KU_2.$$

\subsection{Topological Andr\'e-Quillen homology}
We will first sketch the definition of ordinary Andr\'e-Quillen
homology. See \cite{quillen} for the original account and
\cite{iyengar} for a very readable modern introduction.

\begin{defn}
Let $k$ be a commutative ring with unit and let $A$ be a commutative
$k$-algebra. The \emph{$A$-module of K\"ahler differentials of $A$
over $k$} is the $A$-module generated by elements $d(a)$ for $a \in A$
subject to the relations that $d$ is $k$-linear and satisfies the
Leibniz rule:
$$ d(ab) = d(a)b + ad(b).$$
This $A$-module is denoted by $\Omega^1_{A|k}$.
\end{defn}
The conditions imply $d(1) = d(1\cdot 1) = 2d(1)$ and hence $d(1)=0$.
For a polynomial algebra $A = k[x_1,\ldots,x_n]$ the $A$-module
$\Omega^1_{k[x_1,\ldots,x_n]|k}$ is freely generated by $dx_1, \ldots,
dx_n$. By induction one can show that $d(x_i^m) = mx_i^{m-1}d(x_i)$
for all $m \geq 2$.

Consider for instance the $\F_p$-algebra $\F_p[x]/(x^p-x)$. Then the module of
K\"ahler differentials is generated by $d(x)$. However, as we are in
characteristic $p$ we get
$$ d(x) = d(x^p) = px^{p-1}d(x) =0$$
and hence $\Omega^1_{\F_p[x]/(x^p-x)|\F_p} =0$.

\begin{rem}
For a commutative $k$-algebra $A$ there is an isomorphism between 
$\Omega^1_{A|k}$ and the first Hochschild homology group of $A$ over $k$: 
Every $a \otimes b$ in Hochschild chain degree one is a cycle and if you send 
$a \otimes b$ to $ad(b)$ then this gives a well-defined map modulo Hochschild 
boundaries and it induces an isomorphism $\HH^k_1(A) \cong \Omega^1_{A|k}$ 
\cite[Proposition 1.1.10]{loday}. 
\end{rem}

\begin{defn}
Let $M$ be an $A$-module. A \emph{$k$-linear derivation from $A$ to $M$}
is a $k$-linear map $\delta\colon A \ra M$ which satisfies the Leibniz
rule.
\end{defn}
The set of all such derivations, $\Der_k(A,M)$,
is an $A$-submodule of the $A$-module
of all $k$-linear maps. The symbol $d$ in the definition of
$\Omega^1_{A|k}$ satisfies the conditions of a derivation, hence
the map
$$ d \colon A \ra \Omega^1_{A|k}, \quad a \mapsto da$$
is a derivation, in fact, it is the \emph{universal derivation}:
\begin{prop} \, \cite{iyengar} \,  
For all $A$-modules $M$ the canonical map
$$ \Hom_A(\Omega^1_{A|k}, M) \ra \Der_k(A,M), \quad f \mapsto f \circ d$$
is an $A$-linear isomorphism.
\end{prop}
There is another crucial reformulation of the above isomorphism:
$\Der_k(A,M)$ can also be identified with the morphisms of commutative
$k$-algebras from $A$ to the square-zero extension $A \oplus M$. The latter
is the commutative $A$-algebra with underlying module  $A \oplus M$ with
multiplication
$$ (a_1, m_1)(a_2, m_2) = (a_1a_2, a_1m_2 + a_2m_1), \quad a_1, a_2 \in A,
m_1, m_2 \in M.$$
A derivation $\delta \colon A \ra M$ corresponds to the map into the second
component of $A \oplus M$.

The idea of Andr\'e-Quillen homology is to take the derived funtor of
$A \mapsto M \otimes_A \Omega^1_{A|k}$. But in which sense? As $A$ is a
commutative algebra, we need a resolution of $A$ as such an algebra. The
category of differential graded commutative $k$-algebras in general
doesn't have a (right-induced) model structure, so instead one works with
\emph{simplicial resolutions}. The category of simplicial commutative
$k$-algebras \emph{does} have a nice model structure. Let $P_\bullet \ra A$ be
a cofibrant resolution. Each $P_n$ can be chosen to be a polynomial algebra
\cite[\S 4]{iyengar}.
\begin{defn}
The \emph{Andr\'e-Quillen homology of $A$ over $k$ with coefficients
  in $M$} is
$$ \AQ_*(A|k:M) = \pi_*(M \otimes_{P_\bullet} \Omega^1_{P_\bullet|k}).$$
\end{defn}

A definition of $\Omega^1_{A|k}$ in terms of generators and relations is
not suitable for a generalization to commutative ring spectra. Instead we
use the following description:
\begin{lem}
Denote by $I$ the kernel of the multiplication map
$$ \mu \colon A \otimes_k A \ra A.$$
Then $\Omega^1_{A|k}$ is isomorphic to $I/I^2$.
\end{lem}
\begin{proof}
Note that $I$ is generated by elements of the form
$a \otimes 1 - 1 \otimes a$. Such an element is identified with $d(a)$.
Taking the quotient by $I^2$ corresponds to the Leibniz rule for $d$.
\end{proof}
The ideal $I$  can also be viewed as a non-unital $k$-algebra and
$I/I^2$ is the \emph{module of indecomposables of $I$}.

This definition translates to brave new commutative rings. Basterra's
work is formulated in the setting of \cite{EKMM}:

\begin{defn}
Let $A$ be a commutative $R$-algebra spectrum.
\begin{itemize}
\item
We define
$I(A \wedge_R A)$ as the pullback
$$ \xymatrix{
{I(A \wedge_R A)} \ar@{.>}[r] \ar@{.>}[d] & {A \wedge_R A} \ar[d]^\mu \\
{\ast} \ar[r] & {A}
}$$
\item
If $N$ is a non-unital commutative $R$-algebra spectrum, then its
\emph{$R$-module of indecomposables}, $Q(N)$, is defined as the
pushout
$$ \xymatrix{
{N \wedge_R N} \ar[d]_{\mu_N} \ar[r] & {\ast} \ar@{.>}[d] \\
{N} \ar@{.>}[r] & {Q(N)}
}$$
\item
For an $A$-module spectrum $M$ we define the \emph{topological
Andr\'e-Quillen homology of $A$ over $R$ with coefficients in $M$}
as
\begin{equation} \label{def:taq}
\TAQ(A|R:M) = \mathbf{L}Q(\mathbf{R}I(A \wedge_R A))
\end{equation}
and denote its homotopy groups as $\TAQ_*(A|R:M)$. We abbreviate
$\mathbf{L}Q(\mathbf{R}I(A \wedge_R A))$ by $\Omega_{A|R}$.
\end{itemize}
\end{defn}
Thus for $\Omega_{A|R}$ we take homotopy invariant versions of the kernel of
the multiplication map followed by taking indecomposables by applying
the right derived functor of $I$ and the left derived functor of $Q$.
\begin{defn}
Dually, \emph{topological Andr\'e-Quillen cohomology of $A$ over $R$
with coefficients in $M$} is $F_A(\Omega_{A|R}, M)$ and we set
$\TAQ^n(A|R; M) = \pi_{-n}F_A(\Omega_{A|R}, M)$.
\end{defn}

Basterra proves \cite[Proposition 3.2]{basterra} that
maps from $\Omega_{A|R}$ to $M$ in the homotopy category of $A$-modules
correspond to maps
in the homotopy category of commutative $R$-algebra maps over $A$ from
$A$ to $A \vee M$ where $A \vee M$ carries the square-zero multiplication.

For example, if $f \colon B \ra BGL_1(S)$ is an infinite loop map and $M(f)$
is the associated Thom spectrum, then Basterra and Mandell show
\cite[Theorem 5 and Corollary]{bmthom} that
$$ \TAQ(M(f)) \simeq M(f) \wedge b$$
where $\Omega^\infty b \simeq B$. In the case of an $E_\infty$-space
$B$ the spherical group ring $\Sigma^\infty_+ B$ has
$$ \TAQ(\Sigma^\infty_+ B) \simeq \Sigma_+ B \wedge b.$$
\section{How do we recognize ring spectra as being (non) commutative?}
\label{sec:obs}
If you have a concrete model of a homotopy type, say in symmetric
spectra, then you can be lucky and this model posseses a commutative
structure and you should be able to check this by hand. Of course you
could also try to disprove commutativity by showing that your spectrum
doesn't have power operations as in \eqref{eq:powerops} and this has
been done in many cases, but sometimes you might need a different approach.

\subsection{Obstructions via filtrations and resolutions}
An obstruction theory for $A_\infty$-structures on homotopy ring
spectra was developed as early as 1989 \cite{roainfty} by Alan
Robinson. Obstruction
theories for $E_\infty$-structures came much later: Goerss-Hopkins and
Robinson  \cite{GH}, \cite{robinson} independently developed one with
obstruction groups that later turned out to be isomorphic
\cite{BaRi}. The idea is to use a filtration or resolution of
an operad such that the corresponding filtration quotients or the
corresponding spectral sequence gives rise to obstruction groups that
contain obstructions for lifting a partial structure to a full
$E_\infty$-ring structure \cite[Theorem 5.6]{robinson},
\cite[Corollary 5.9]{GH}. The Goerss-Hopkins approach also allows to
calculate the homotopy groups of the derived $E_\infty$ mapping space
between two such $E_\infty$-ring spectra \cite[Theorem 4.5]{GH}.

The obstruction groups have as input the algebra of cooperations $E_*E$ of
a spectrum $E$ and they compute Andr\'e-Quillen cohomology groups of the
graded commutative $E_*$-algebra $E_*E$ in the setting of differential graded
(or simplicial)
$E_\infty$-algebras. See \cite{mandell-aqeinfty}, \cite[\S 2.4]{GHall}
for background on these
cohomology groups and see \cite[\S 2]{BaRi} for the comparison results. 
In Robinson's setting these groups are called Gamma-cohomology. 
The obstruction groups vanish if for instance $E_*E$ is \'etale as an
$E_*$-algebra.

If you prefer to work with explicit chain complexes, then there are several
equivalent ones computing Gamma cohomology groups in Robinson's setting
\cite[\S 2.5]{robinson},
\cite[\S 6]{RobWh}, \cite[\S 2]{piri} and therefore, by the comparison result 
from \cite[Theorem 2.6]{BaRi}, computing the 
obstruction groups in the Goerss-Hopkins setting as well.

There is another version of obstruction theory for promoting a homotopy
$T$-algebra structure to an actual one by Johnson and Noel
\cite{johnson-noel} where $T$ is a monad. This includes obstructions for
operadic structures on spectra but also includes for instance group actions.
Noel shows that in certain situations the obstruction theory
\cite{johnson-noel} can be compared to the one of \cite{GH}.

We list some important applications:
\begin{enumerate}
\item
The development of the Hopkins-Miller and Goerss-Hopkins obstruction theory
was motivated by the Morava-$E$-theory spectra, also known as
Lubin-Tate spectra, $E_n$ and variants of those. These are Landweber
exact cohomology theories that govern the deformation theory of height $n$
formal group laws. In \cite{rezk} an obstruction theory was established that
allowed to show that the $E_n$ are $A_\infty$ spectra and that the
Morava stabilizer group $\mathbb{G}_n$ acts on $E_n$ via maps of
$A_\infty$ spectra. In \cite{GH} the corresponding obstruction theory
for $E_\infty$-structures was developed and \cite[Corollaries 7.6, 7.7]{GH}
shows that the $\mathbb{G}_n$-action is via $E_\infty$-maps.
\item
It was known that $\KU$ and $\KO$ are $E_\infty$-spectra and it was also
known that the $p$-completed Adams summand $L_p$ is $E_\infty$.
In \cite{br} Andy Baker and I use Robinson's version of the
$E_\infty$-obstruction
theory to show that these $E_\infty$-structures are unique and that the
$p$-local Adams summand also has a unique $E_\infty$-structure.
Uniqueness also holds for the connective covers \cite{br-covers}. It is 
important to have uniqueness results for $E_\infty$-structures because 
calculations can depend on a choice of such a structure. 
\item
For an $E_\infty$-ring spectrum $R$ there
is a $\theta$-algebra structure on its $p$-adic K-theory,
$\pi_*L_{K(1)}(KU_p \wedge R)$ \cite[Theorem 2.2.4]{GHall} and in good cases 
$$\pi_*L_{K(1)}(KU_p \wedge R) \cong \lim_k (KU_p)_*(R \wedge M(p^k))$$ 
where $M(p^k)$ is the mod-$p^k$ Moore spectrum. The study of such
structures was initiated by McClure in \cite[Chapter IX]{hinfty}. 
There is a variant of the Goerss-Hopkins obstruction theory for realizing for
instance a $\theta$-algebra  
(see \cite[\S 2.4.4]{GHall} and \cite[Theorem 5.14]{ln1}) as a $K(1)$-local
$E_\infty$-ring spectrum. 

There is one for realizing an $E_\infty$-$H\!k$-algebra spectrum
with a fixed Dyer-Lashof structure on its homotopy \cite[Proposition
2.2]{noel} (for $k$ a field of characteristic $p$).  Other variants
can be found in the literature.  

The $\theta$-algebra version was successfully applied by Lawson and
Naumann  \cite{ln1} to
show that $\BP\langle 2\rangle$ at $2$ has an $E_\infty$-structure. By a
different method Hill and Lawson \cite[Theorem 4.2]{hl} find a commutative
model for $\BP\langle 2\rangle$ at the prime $3$. 
\item
Mathew, Naumann and Noel use operations in Morava-$E$-theory to prove 
May's nilpotence conjecture: 
\begin{thm} \, \cite[Theorem A]{mnn} \,
If $R$  is an $H_\infty$-ring spectrum and if $x \in \pi_*(R)$ is in the 
kernel of the Hurewicz homomorphism  $\pi_*(R) \ra H_*(R; \Z)$, then $x$ 
is nilpotent. 
\end{thm}
They use this -- among many other applications -- for the following result 
about $E_\infty$-ring spectra:
\begin{thm}  \, \cite[Proposition 4.2]{mnn} \,
If $R$ is an $E_\infty$-ring spectrum and if there is an $m \in \Z$, $m 
\neq 0$ with $m\cdot 1=0 \in \pi_0(R)$, then for all primes $p$ and for all 
$n \geq 1$: 
$$ K(n)_*(R) \cong 0.$$
\end{thm}
Lawson observed  that using $K(n)$-techniques (see \cite{ravenelnil} for 
background) this implies that for finite $E_\infty$-ring spectra $R$ either 
the rational homology is non-trivial or $R$ is weakly contractible, because if 
$H_*(R; \Q) \cong 0$, then by the above result all the Morava K-theories also 
vanish on $R$, but then the finiteness of $R$ implies weak 
contractibility (see \cite[Remark 4.3]{mnn} for the full argument).   
\end{enumerate}

The Dyer-Lashof variant is for instance important when one wants to decide
whether a given $H_\infty$-map can be upgraded to an $E_\infty$-map: Roughly
speaking, an $H_\infty$-spectrum is like an $E_\infty$-spectrum in the
homotopy category.  
You can find applications of
this approach for instance in Noel's work \cite{noel} and in
\cite{johnson-noel}.

Other spectra, for instance, $\BP$, come with homology operations just because
they sit in the right place: analyzing the maps $\MU \ra \BP \ra H\F_p$ 
gives \cite[p.~63]{hinfty} that $(H\F_p)_*(\BP)$ embeds into the dual 
of the Steenrod algebra such that  $(H\F_p)_*(\BP)$ is closed under the 
action of the Dyer-Lashof algebra -- even without establishing a
structured multiplication on $\BP$.  This led Lawson \cite{lawson} to look
for the right obstructions for an $E_\infty$-structure of $\BP$ at $2$ via 
secondary operations (see Theorem \ref{thm:bp}).

\subsection{Obstructions via Postnikov towers}
A different approach to obstruction theory is to consider Postnikov
towers in the world of commutative ring spectra \cite{basterra} or in
the setting of $E_n$-algebras \cite{bama}.

To this end Basterra uses $\TAQ$-cohomology to lift ordinary $k$-invariants of
a connective commutative ring spectrum to $k$-invariants in a multiplicative
Postnikov tower:

Assume that $R$ is a connective commutative ring spectrum. Then there is a
map of commutative ring spectra
$$ p_0 \colon R \ra H(\pi_0(R))$$
and without loss of generality we can assume that $p_0$ is a cofibration of
commutative ring spectra that realizes the identity on $\pi_0$, 
\ie, $\pi_0(p_0) = \id_{\pi_0(R)}$.

If we abbreviate $\pi_0(R)$ to $B$ and if $M$ is a $B$-module,
then an element in $\TAQ^n(A|R; H\!M)$ corresponds to a morphism
$\varphi \colon A \ra A\vee \Sigma^n H\!M$ in the homotopy category of
$R$-algebra spectra over $A$ and we can form the pullback of
$$\xymatrix{
{} & {A} \ar[d]^{i_A} \\
{A} \ar[r]^(0.4)\varphi & {A \vee \Sigma^n H\!M}
}$$
If we postcompose $\varphi$ with the projection map to $\Sigma^nH\!M$
\begin{equation} \label{eq:forgetful}
 \xymatrix@1{{A} \ar[r]^(0.4)\varphi & {A \vee  \Sigma^n H\!M} \ar[r] &
{\Sigma^n H\!M}   }
\end{equation}
such a $\TAQ$-class forgets to an $\Ext$-class in
$\Ext^n_R(A; H\!M)$, \ie, if $R$ is the sphere spectrum, 
to an ordinary cohomology class. Basterra shows that
this projection maps  $k$-invariants in the world of
commutative ring spectra to ordinary $k$-invariants of the
underlying spectrum. 
\begin{thm} \, \cite[Theorem 8.1]{basterra} \,
For any connective commutative ring spectrum $A$ there is a sequence of
commutative ring spectra $A_i$, $\pi_0(A)$-modules $M_i$ and elements
$$ \tilde{k}_i \in \TAQ^{i+2}(A_{i}|S; H\!M_{i+1})$$
such that
\begin{itemize}
\item
$A_0 = H\pi_0(A)$ and $A_{i+1}$ is the pullback of $A_i$ with
respect to $\tilde{k}_i$,
\item
$\pi_jA_i =0$ for all $j >i$, 
\item
there are maps of commutative ring spectra $\lambda_i \colon A \ra
A_i$ which induce an isomorphism in homotopy groups up to degree $i$ 
such that the diagram 
$$\xymatrix{
{} & {A_{i+1}} \ar[d] \\
{A} \ar[r]^{\lambda_i} \ar[ur]^{\lambda_{i+1}} & {A_i}
}$$
commutes in the homotopy category of commutative ring spectra.
\end{itemize}
\end{thm}
You start with $A_0 = H\pi_0(A)$ and then you have to find a suitable
map $A_0 \ra A_0 \vee
\Sigma^2H(\pi_1(A))$ as a starting point of the multiplicative Postnikov tower.

Basterra's result can be used as an obstruction theory as follows. If
$A$ is a connective spectrum then it has an ordinary Postnikov tower with
$k$-invariants living in ordinary cohomology groups
$$ k_{i} \in  H^{i+2}(A_{i}; \pi_{i+1}(A)).$$
You can then investigate
whether it is possible to find multiplicative $k$-invariants
$$ \tilde{k}_{i} \in  \TAQ^{i+2}(A_{i}|S; H\pi_{i+1}(A))$$
that forget to the $k_{i}$'s under the map from \eqref{eq:forgetful}.

Basterra and Mandell show  the following partial result using Postnikov
towers for $E_n$-algebra spectra.
\begin{thm} \, \cite[Theorem 1.1]{bama} \,
The Brown Peterson spectrum, $\BP$, has an $E_4$-structure at every prime.
\end{thm}
This ensures by the main result of \cite{mandellderived} that the
derived category of $\BP$-module spectra has a symmetric monoidal smash product.
Tyler Lawson, however, showed that there are certain secondary
operations in the $\mathbb{F}_2$-homology of every such spectrum with an
$E_{12}$-structure and he could show that these are not present in the
$\mathbb{F}_2$-homology of $\BP$ at $2$. Let $\BP\langle n\rangle$
denote the spectrum $BP/(v_{n+1},v_{n+2},\ldots)$. 
\begin{thm} \, \cite[Theorem 1.1.2]{lawson} \, \label{thm:bp}
The Brown-Peterson spectrum at the prime $2$ does not possess an
$E_n$-structure for any $n$ with $12 \leq n \leq \infty$. The
truncated Brown-Peterson spectrum $\BP\langle n\rangle$ for $n \geq 4$
cannot have an
$E_n$-structure for any $n$ with $12 \leq n \leq \infty$.
\end{thm}

\subsection{Realization of $E_\infty$-spectra via derived algebraic geometry}
There is a completely different important and highly successful approach to
realization problems, namely using \emph{derived algebraic
  geometry}. You can learn about derived algebraic geometry in the
next Chapter of the book.

\section{What are \'etale maps?} \label{sec:etale}
We first recall the algebraic notion of an \'etale $k$-algebra from
\cite[E.1]{loday}: Let $k$ be a commutative ring and let $A$ be a
finitely generated commutative
$k$-algebra. Then $A$ is \emph{\'etale}, if $A$ is flat
over $k$ and if the module of K\"ahler differentials $\Omega^1_{A|k}$ is
trivial. If $\Omega^1_{A|k}=0$, then $k \ra A$ is called \emph{unramified}.
A $k$-algebra $B$ (not necessarily commutative) is called \emph{separable},
if the multiplication map
$$ B \otimes_k B^o \ra B$$
has a section as a $B$-bimodule map. In algebra, a commutative separable
algebra has Hochschild homology concentrated in homological degree zero, in
particular the module of K\"ahler differentials is trivial.

\subsection{Rognes' Galois extensions of commutative ring spectra}
\begin{defn} \label{def:galois} \, \cite[Definition 4.1.3]{rognes} \,
Let $A \ra B$ be a map of commutative ring spectra and let $G$ be a
finite group acting on $B$ via commutative $A$-algebra maps. Assume
that $S \ra A \ra B$ is a sequence of cofibrations in the model
structure on commutative ring spectra of \cite[Corollary
VII.4.10]{EKMM}. Then $A \ra B$ is a $G$-Galois extension if
\begin{enumerate}
\item
the canonical map $\iota \colon A \ra B^{hG}$ is a weak equivalence
and
\item
\begin{equation} \label{eq:unram} h \colon B \wedge_A B \ra \prod_G
  B \end{equation} is a weak equivalence.
\end{enumerate}
\end{defn}
The first condition is the familiar fixed points condition from
classical Galois theory of fields. The map $\iota$ comes from taking
the adjoint of the map
$$\xymatrix@1{{A \wedge EG_+} \ar[r]^{\id \wedge p} & {A \wedge S^0
  \cong A} \ar[r] & B} $$
where $p \colon EG_+ \ra S^0$ collapses $EG$ to the non-base point of
$S^0$.

The map $h$ is adjoint to the composite
$$ B \wedge_A B \wedge G_+ \ra B \wedge_A B \ra B$$
that comes from the $G$-action on the right factor of $B \wedge_A B$
followed by the multiplication in $B$. (Informally, if smashes were
tensors, then $h(b_1 \otimes b_2) = (b_1\cdot g(b_2))_{g \in G}$.)
Note that $\prod_G B$ is isomorphic to $F(G_+, B)$, so we could rewrite
the condition in \eqref{eq:unram} as the requirement that
$$ h \colon B \wedge_A B \ra F(G_+, B)$$
is a weak equivalence.

The condition that the map $h$ from \eqref{eq:unram} is a
weak equivalence is crucial. It is
also necessary for Galois extensions of
discrete commutative rings in order to ensure that the extension is
unramified. For instance $\Z \subset \Z[i]$ satisfies that
$\Z[i]^{C_2} = \Z$, but $h \colon \Z[i] \otimes_\Z \Z[i] \ra \Z[i]
\times \Z[i]$ is not surjective: $h$ detects the ramification at
the prime $2$. Therefore $\Z \ra \Z[i]$ is \emph{not} a $C_2$-Galois 
extension but $\Z[\frac{1}{2}] \ra \Z[\frac{1}{2},i]$ is $C_2$-Galois.

Galois extensions of commutative ring spectra can have rather bad
properties as modules. So the following definition is actually an
additional assumption (this does not happen in the discrete setting).

\begin{defn} \cite[Definition 4.3.1]{rognes}
A Galois extension $A \ra B$ is \emph{faithful} if it is faithful
as an $A$-module, \ie, for every $A$-module
$M$ with $M \wedge_A B \simeq *$ we have $M \simeq *$.
\end{defn}

Important examples of Galois extensions of commutative ring spectra are the
following. By $C_n$ we denote the cyclic group of order $n$.

\begin{enumerate}
\item
The concept of Galois extensions of commutative ring spectra
corresponds to the one for commutative rings via the
Eilenberg-Mac\,Lane spectrum functor
\cite[Proposition 4.2]{rognes}:
Let $R \ra T$ be a homomorphism of discrete commutative rings and let
$G$ be a finite group acting on $T$ via $R$-algebra
homomorphisms. Then $R \ra T$ is a $G$-Galois extension of commutative
rings if and only if $\HR \ra H\!T$ is a $G$-Galois extension of commutative
ring spectra.
\item
The complexification of real vector bundles gives rise to a map of
commutative ring spectra $\KO \ra \KU$ from real to complex topological
K-theory. There is a $C_2$-action on $\KU$ corresponding to complex
conjugation of complex vector bundles. Rognes shows \cite[Proposition
5.3.1]{rognes} that this turns $\KO \ra \KU$ into a $C_2$-Galois
extension.
\item
At an odd prime $p$ there is a $p$-adic Adams operation on $KU_p$ that gives
rise to a $C_{p-1}$-action on $KU_p$ such that $L_p \ra KU_p$ is a
$C_{p-1}$-Galois extension (see \cite[5.5.4]{rognes}). 
\item
There is a notion of pro-Galois extensions of commutative ring spectra
and $L_{K(n)}S \ra E_n$ is a $K(n)$-local pro-Galois extension with the
extended
Morava stabilizer group as the Galois group \cite[Theorem 5.4.4]{rognes}.
\item
Let $p$ be an arbitrary prime.
The projection map $\pi\colon EC_p \ra BC_p$ induces a map on function spectra
$$ F(\pi_+, H\F_p) \colon F((BC_p)_+, H\F_p) \ra F((EC_p)_+, H\F_p)
\sim H\F_p$$
which identifies $H\F_p$ as a $C_p$-Galois extension over $F((BC_p)_+,
H\F_p)$ \cite[Proposition 5.6.3]{rognes}. Hence in the world of
commutative ring spectra group cohomology sits between $S$ and
$H\F_p$ as the base of a Galois extension! Beware, this Galois extension is not
faithful. This observation is due to Ben Wieland: the Tate construction
$H\F_p^{tC_p}$ isn't trivial and it is  actually killed by the Galois extension
(in the spectral sequence you augment a Laurent generator to zero).
\item
Studying elliptic curves with level structures gives $C_2$-Galois
extensions $\TMF_0(3) \ra \TMF_1(3)$ and $\Tmf_0(3) \ra \Tmf_1(3)$
\cite[Theorems 7.6, 7.12]{mm}. For  $\TMF_1(3)$ (and  $\Tmf_1(3)$) you consider
elliptic  curves  with  one  chosen  point  of  exact
order $3$ and for $\TMF_0(3)$ (and  $\Tmf_0(3)$) you only remember
a subgroup of order $3$. As $C_2 \cong \Z/3\Z^\times$ this gives a
$C_2$-action. This can be made rigorous (see \cite{hm,hl,mm}).

\end{enumerate}

\subsection{Notions of \'etale morphisms}
Weibel-Geller \cite{wg} show that for an \'etale extension of
commutative rings $\varphi\colon A \ra B$ Hochschild homology
satisfies \emph{\'etale descent}: The map $\HH(\varphi)_*$
induces an isomorphism
\begin{equation} \label{eq:edescent-alg}
B \otimes_{A} \HH_*(A) \cong \HH_*(B)
\end{equation}
and for finite $G$-Galois extensions $\varphi\colon A \ra B$ one
obtains \emph{Galois descent}:
\begin{equation} \label{eq:gdescent-alg}
\HH_*(A) \cong \HH_*(B)^G
\end{equation}
It is easy to see that for a $G$-Galois extension of discrete commutative 
rings $\varphi\colon A \ra B$ with finite $G$, the induced extension of 
graded commutative rings $\HH_*(\varphi) \colon \HH_*(A) \ra \HH_*(B)$ is 
again $G$-Galois. In addition to having the right fixed-point property it 
satisfies 
\begin{align*} 
\HH_*(B) \otimes_{\HH_*(A)} \HH_*(B) & \cong B \otimes_{A} \HH_*(A) 
\otimes_{\HH_*(A)} B \otimes_{A} \HH_*(A) \\
& \cong B \otimes_A B \otimes_A \HH_*(A) \\
& \cong \prod_G B \otimes_A \HH_*(A) \\
& \cong \prod_G \HH_*(B).
\end{align*}

If $\varphi \colon A \ra B$ is \'etale, then the module of K\"ahler
differentials $\Omega^1_{B|A}$ is trivial and it can be easily seen
that the map $B \ra \HH^A_*(B)$  is an isomorphism and that
Andr\'e-Quillen homology of $B$ over $A$ is trivial, because
\'etale algebras are smooth.

For commutative ring spectra the situation is different. There are
several non-equivalent notions of \'etale maps:

\begin{defn}
Let $\varphi\colon A \ra B$ be a morphism of commutative ring
spectra.
  \begin{enumerate}
  \item \, \cite[Definition 7.5.1.4]{lurie-HA}
We call $\varphi$
\emph{Lurie-\'etale}, if $\pi_0(\varphi) \colon \pi_0(A) \ra \pi_0(B)$
is an \'etale map of commutative rings and if the canonical map
$$ \pi_*(A) \otimes_{\pi_0(A)} \pi_0(B) \ra \pi_*(B)$$
is an isomorphism.
\item \, \cite[Definiton 3.2]{mm-hkr}, \cite[Definition 9.2.1]{rognes}
The morphism $\varphi$ is \emph{(formally) $\THH$-\'etale}, if $B \ra
\THH^A(B)$ is a weak equivalence.
\item \,  \cite[Definiton 3.2]{mm-hkr}, \cite[Definition 9.4.1]{rognes}
We define $\varphi$ to be \emph{(formally) $\TAQ$-\'etale}, if
$\TAQ(B|A)$ is weakly equivalent to $\ast$.
  \end{enumerate}
\end{defn}
\begin{rem}
\begin{itemize}
\item[]
\item
Rognes \cite{rognes} reserves the labels $\THH$-\'etale and
$\TAQ$-\'etale only for such maps that -- in addition to the
conditions above -- identify $B$ as a dualizable
$A$-module.
\item
The condition of being Lurie-\'etale is strong and is a very algebraic one.
It is for instance not
satisfied by the $C_2$-Galois extension $\KO \ra \KU$ because on the
level of homotopy groups this extension is rather appalling, compare
\eqref{eq:kotoku}.
\item
McCarthy and Minasian show that $\THH$-\'etale implies $\TAQ$-\'etale
and they show that for $n > 1$ the map $H\F_p \ra F(K(\Z/p\Z, n)_+, H\F_p)$ is
a $\TAQ$-\'etale morphism that is not 
$\THH$-\'etale. They attribute this example to Mike Mandell
\cite[Example 3.5]{mm-hkr}. Minasian \cite[Corollary 2.8]{minasian}
proves that both notions are 
equivalent for morphisms between connective commutative ring spectra.
\item
For connective spectra, the notion of Lurie-\'etaleness  has
several good features \cite[\S 7.5]{lurie-HA} and Mathew shows  in
\cite[Corollary 3.1]{Mat} that one can use \cite[Lemma 8.9]{lurie-dagvii}
to show that under some finiteness condition $\TAQ$-\'etaleness implies
Lurie-\'etaleness in the connective case.
\end{itemize}
\end{rem}

\begin{defn} \cite[Definition 9.1.1]{rognes}
Let $C$ be a cofibrant associative $A$-algebra spectrum. Then $C$ is
\emph{separable} if 
the multiplication map $\mu\colon C \wedge_A C^o \ra C$ has a section in
the homotopy category of $C$-bimodule spectra.
\end{defn}

\begin{prop} \, \cite[9.2.6]{rognes} \, 
If $C$ is a commutative separable $A$-algebra spectrum, then $C$ is
$\THH$-\'etale. 
\end{prop}
\begin{proof}
Recall from Remark \ref{rem:augmentationTHH} that $\THH^A(C)$ is an augmented
commutative $C$-algebra spectrum, so the composite of the unit map $C
\ra \THH^A(C)$ with the augmentation
$$ C \ra \THH^A(C) \ra C$$
is the identity. In addition, we get a splitting in the homotopy
category of $C$-bimodule spectra,
$$ \xymatrix@1{{C} \ar[r]^(0.4)s & {C \wedge_A C} \ar[r]^\mu & {C,}} $$
\ie, the above composite is the identity on $C$. Taking the derived smash
product $C \wedge^L_{C \wedge_A C} (-)$ of the above sequence gives the
sequence
$$ \THH^A(C) \ra C \ra \THH^A(C)$$
in which the last map is equivalent to the unit map of $\THH^A(C)$ and
whose composite is the identity. So the unit map $C \ra \THH^A(C)$ has a
right and a left inverse in the homotopy category of $C$-module spectra.
\end{proof}

\begin{defn}
Let $A \ra B$ be a map of commutative ring spectra and let $G$ be a 
finite group acting on $B$ via maps of commutative $A$-algebra spectra. Assume
that $S \ra A \ra B$ is a sequence of cofibrations in the model
structure on commutative ring spectra of \cite[Corollary
VII.4.10]{EKMM}. Then $A \ra B$ is \emph{unramified} if
$$ h \colon B \wedge_A B \ra \prod_G
  B $$
is a weak equivalence.
\end{defn}

\begin{prop} \, (compare \cite[Lemma 9.1.2]{rognes}) \,
If $A \ra B$ is unramified, then $B$ is separable over $A$.
\end{prop}
\begin{proof}
The canonical inclusion map $i \colon B \ra F(G_+, B)$ can be modelled by the
pointed map from $G_+$ to $S^0$, that sends the neutral element $e \in G$ to
the non-basepoint of $S^0$ and sends all other elements to the basepoint. We
define a section to the multiplication map of $B$ to be 
$$ \xymatrix@1{{B} \ar[r]^(0.4)i & {F(G_+, B)} & {B \wedge_A B.}
\ar[l]_{h, \sim}}$$
Note that $h$ is not a $B$-bimodule map, but we are only interested in its
$e$-component of $F(G_+, B)$.
\end{proof}
Thus we can conclude that unramified maps of commutative ring spectra
are $\THH$-\'etale and that the failure of the map $B \ra \THH^A(B)$
to being a weak equivalence detects ramification. This idea was
exploited in \cite{dlr2} in order to identify for instance $ko \ra ku$
as a wildly ramified extension whereas the inclusion of the Adams summand
$\ell \ra ku_{(p)}$ is tamely ramified \cite[Theorems 4.1, 5.2]{dlr2}. Sagave 
also identified this map as being log-\'etale \cite[Theorem 1.6]{sagave}. 

\subsection{Versions of \'etale descent}
Transferring the Geller-Weibel result to the setting of commutative
ring spectra, it seems natural to define two versions of descent:

\begin{defn}
In the following $\varphi\colon A \ra B$ is a cofibration and $A$ is
cofibrant.
  \begin{itemize}
  \item
The morphism $\varphi \colon A \ra B$ satisfies \emph{\'etale descent}
if the canonical morphism
\begin{equation} \label{eq:etaledescent}
B \wedge_A \THH(A) \ra \THH(B)
\end{equation}
is a weak equivalence.
\item
If $\varphi \colon A \ra B$ is a map of commutative ring spectra and
if $G$ is a
finite group acting on $B$ via commutative $A$-algebra maps, then we
say that $\varphi$ satisfies \emph{Galois descent} if the map
\begin{equation} \label{eq:galoisdescent}
\THH(A) \ra \THH(B)^{hG}
\end{equation}
is a weak equivalence.
\end{itemize}
\end{defn}
Akhil Mathew clarifies the relationship between the different notions
of \'etale morphisms and the notions of descent. He
proves that Lurie-\'etale morphisms satisfy \'etale descent
\cite[Theorem 1.3]{Mat} and he shows that for a faithful $G$-Galois
extension with finite Galois group $G$, both descent properties are
equivalent \cite[Proposition 4.3]{Mat} and they are in turn equivalent
to the property that $\THH(A) \ra \THH(B)$ is again a $G$-Galois extension.

Moreover, he shows that the morphism
$$ \varphi \colon F(\mathbb{S}^1_+, H\F_p) \ra F(\mathbb{S}^1_+,
H\F_p)$$ that is induced by the degree-$p$ map on $\mathbb{S}^1$ is a
faithful $C_p$-Galois extension, but that it does \emph{not} satisfy \'etale
descent \cite[Theorem 2.1]{Mat} and hence it doesn't satisfy Galois descent.

The Hopf fibration $\mathbb{S}^1 \ra \mathbb{S}^3 \ra \mathbb{S}^2$ is
a principal $\mathbb{S}^1$-bundle. The corresponding morphism of 
commutative $H\Q$-algebra spectra of cochains
$$ F(\eta, H\Q) \colon  F(\mathbb{S}^2_+, H\Q) \ra F(\mathbb{S}^3_+, H\Q)$$
is therefore an $\mathbb{S}^1$-Galois extension \cite[Proposition
5.6.3]{rognes}.

In joint work with Christian Ausoni we show that the morphism $F(\eta,
H\Q)$  does  not satisfy Galois descent, \ie, 
$$ \THH(F(\mathbb{S}^2_+, H\Q)) \nsim \THH(F(\mathbb{S}^3_+, 
H\Q))^{h\mathbb{S}^1}:$$ 
The homotopy groups of $\THH(F(\mathbb{S}^2_+, H\Q))$ contain an 
element in degree $-1$ that is not present in
$\pi_*(\THH(F(\mathbb{S}^3_+, 
H\Q))^{h\mathbb{S}^1})$. %%%might not be faithful

Mathew identifies the problem with \'etale descent of finite faithful Galois
extensions for $\THH$ as being
caused by the non-trivial fundamental group of $\mathbb{S}^1$. He
shows the following result.
\begin{thm} \, \cite[Proposition 5.2]{Mat} \, 
Let $X$ be a simply connected pointed space and let $A \ra B$ be a
faithful $G$-Galois extension of commutative ring spectra with finite
$G$. Then the map
$$ B \wedge_A (X \otimes A) \ra X \otimes B$$
is a weak equivalence.
\end{thm}
In particular, higher order topological Hochschild homology,
$\THH^{[n]}$  for $n \geq 2$, \emph{does} satisfy \'etale descent for
faithful finite Galois extensions. However, \'etale descent remains
for instance an issue for torus homology.

Note that sometimes $\THH$ \emph{does} satisfy descent, even for ramified 
maps of commutative ring spectra. For instance, Christian Ausoni shows in 
\cite[Theorem 10.2]{ausoni} that $\THH(\ell_p)$ is $p$-adically equivalent 
to $\THH(ku_p)^{hC_{p-1}}$ and even that $K(\ell_p)$ is $p$-adically 
equivalent to $K(ku_p)^{hC_{p-1}}$. 

\begin{rem}
In \cite{cmnn} Clausen, Mathew, Naumann and Noel prove far-reaching 
Galois descent results for topological
Hochschild homology and algebraic K-theory; in particular they confirm
a Galois descent conjecture for algebraic K-theory by Ausoni and Rognes
in many important cases.
They identify $\THH$ as a \emph{weakly additive invariant} 
(see \cite[Definition 3.11]{cmnn})  and prove descent in the form of
\cite[Theorems 5.1,5.6]{cmnn}.
\end{rem}
\section{Picard and Brauer groups} \label{sec:pic}
\subsection{Picard groups in the setting of a symmetric monoidal category}
Let $(\C, \otimes, 1, \tau)$ be a symmetric monoidal category. An important
class of objects in $\C$ are those objects $C$ that have an inverse with
respect to $\otimes$, \ie, there is an object $C'$ of $\C$ such that
$$ C \otimes C' \cong 1.$$

One wants to gather such objects in a category and build a space and
spectrum out of them:
\begin{defn}
The \emph{Picard groupoid of $\C$}, $\Picard(\C)$, is the category whose
objects are the invertible objects of $\C$ and whose morphisms are
isomorphisms between invertible objects.
\end{defn}
In general, this category might not be small.
Note that if $C_1$ and $C_2$ are objects of $\Picard(\C)$, then so is
$C_1 \otimes C_2$; in fact, $\Picard(\C)$ is itself a symmetric monoidal
category.
\begin{defn}
Let $\C$ be as above and assume that $\Picard(\C)$ is small.

Then $\PIC(\C)$ is the classifying space of the symmetric
monoidal category $\Picard(\C)$ and let $\pic(\C)$ denote the connective
spectrum associated to the infinite loop space associated to $\PIC(\C)$.

The \emph{Picard group of $\C$}, $\Pic(\C)$, is $\pi_0\PIC(\C)$.
\end{defn}
If the Picard groupoid of $\C$ is small, then the Picard group can also be
described as the set of isomorphism classes
of invertible objects of $\C$ with the product
$$ [C_1] \otimes [C_2] := [C_1 \otimes C_2].$$
The neutral element is the isomorphism class of the unit, $[1]$.

\begin{defn}
Let $R$ be a (discrete) commutative ring, then we denote by $\Pic(R)$
the Picard group of the symmetric monoidal category of the category
of $R$-modules and
by $\PIC(R)$ (and $\pic(R)$) the Picard space (and Picard spectrum) of this
category.
\end{defn}
For instance the Picard group of a ring of integers in a number
ring is its ideal class group.

\subsection{Picard group for commutative ring spectra}
For commutative ring spectra $R$, the above definition of $\PIC(R)$ and
$\pic(R)$ would either be much too rigid (if one would choose $\C$ to be the
category of $R$-module spectra and isomorphisms) or not strict enough 
(if one would take $\C$ to be the homotopy category of $R$-module spectra).
See \cite[\S 2]{ms-picard} for an adequate background for a suitable
definition and see
\cite[\S 4]{g-lawson} for a dictionary how to pass from a commutative ring
spectrum
$R$ and its category of modules to the $\infty$-categorical setting. Instead
of working with symmetric monoidal categories, one uses presentable symmetric
monoidal $\infty$-categories $\C$. Then the Picard $\infty$-groupoid of $\C$
is the maximal subgroupoid of the underlying $\infty$-category of $\C$ spanned
by the invertible objects. This groupoid is equivalent to a grouplike
$E_\infty$-space $\PIC(\C)$  and hence there is a
connective ring spectrum, $\pic(\C)$, associated to $\C$ \cite[\S 5]{g-lawson}.

Let $R$ be a commutative ring spectrum.
The operadic nerve of the category of
cofibrant-fibrant $R$-modules is a stable presentable symmetric monoidal
$\infty$-category \cite[Proposition 4.1.3.10]{lurie-HA} and we will
abbreviate this as the $\infty$-category of $R$-modules, $R\text{mod}$.

\begin{defn}
The \emph{Picard group of a commutative ring spectrum $R$, $\Pic(R)$}, is the 
group $\pi_0(\PIC(R\text{mod}))$.
\end{defn}
Again, these Picard groups can also be described as the set of isomorphism
classes of invertible $R$-modules in the homotopy category of $R$-module
spectra.

The Picard space $\PIC(R)$ is a delooping of the units of $R$ 
(\cite[\S 2.2]{ms-picard}, \cite[\S 5]{szymik}): There is an equivalence  
$$ \PIC(R) \simeq \Pic(R) \times BGL_1(R).$$ 

\begin{rem}
There is a map $\Pic(\pi_*R) \ra \Pic(R)$ that realizes an element in the
algebraic Picard group of invertible graded $\pi_*R$-modules as a module
over $R$ and in many cases this map is actually an isomorphism
\cite[Theorem 43]{br-picard}. In this case we call $\Pic(R)$
\emph{algebraic}.  A notable exception comes from Galois extensions of ring
spectra: As in algebra, if $A \ra B$ is a
$G$-Galois extension of commutative ring spectra with abelian Galois group
$G$, then $[B] \in \Pic(A[G])$ \cite[Proposition 6.5.2]{rognes}. But for
instance $[KU_*]$ is certainly \emph{not} an element in the algebraic Picard
group $\Pic(KO_*[C_2])$, see \eqref{eq:kotoku}.
\end{rem}
The equivalence classes of suspensions of $R$ are always in $\Pic(R)$, but if
$R$ is periodic, these suspensions don't generate a free abelian group.
Let us mention some crucial examples of Picard groups of commutative ring
spectra:
\begin{itemize}
\item
The Picard group of the initial commutative ring spectrum $S$ is
$\Pic(S) \cong \Z$ where $n \in \Z$ corresponds to the class of $S^n$
\cite{hms}.
\item
For connective commutative ring spectra the Picard group of $R$ is
algebraic \cite[Theorem 21]{br-picard}, \cite[Theorem
2.4.4]{ms-picard}. 
\item
For periodic real and complex K-theory the Picard groups just notices the
suspensions of the ground ring: The Picard group of $KU$ is
algebraic: $\Pic(\KU) \cong \Z/2\Z$, and
$\Pic(\KO) \cong \Z/8\Z$ (Hopkins, \cite[Example 7.1.1]{ms-picard} and
\cite[\S 7]{g-lawson}).
\item
The same applies to the periodic verison of the spectrum of topological
modular forms: $\Pic(\TMF) \cong \Z/576\Z$ \cite[Theorem A]{ms-picard}. But for
$\Tmf$, the spectrum of
topological forms that mediates between $\TMF$ and its connective version
$\tmf$ one gets \cite[Theorem B]{ms-picard}
$$ \Pic(\Tmf) \cong \Z \oplus \Z/24\Z$$
where the copy of the integers comes from the suspensions of $\Tmf$ and the
generator of the $\Z/24\Z$-summand is described in
\cite[Construction 8.4.2]{ms-picard}
\item
For any odd prime and any finite subgroup $G$ of the full Morava stabilizer
group $\mathbb{G}_{p-1}$ Heard, Mathew and Stojanoska
\cite[Theorem 1.5]{hms-picard} prove -- using Galois
descent techniques for $\pic$ -- that the Picard group of $E_{p-1}^{hG}$ is
a cyclic group generated by the suspension of $E_{p-1}^{hG}$.
\end{itemize}
A Picard group that contains more elements than just the ones coming from
suspensions of the commutative ring spectrum says that there are more
self-equivalences of the homotopy category of $R$-modules than the standard
suspensions. One might view this as twisted suspensions. Gepner and Lawson
explore the concept of having a Picard-grading on the category of $R$-module 
spectra and they develop a Pic-resolution model category structure in the sense 
of Bousfield \cite[\S 3.2]{g-lawson}.

\subsection{Descent method and local versions}
A crucial method for calculating Picard groups is Galois descent. If
$A \ra B$ is a $G$-Galois extension (for $G$ finite), then for the Picard
spectra and spaces the following equivalences hold \cite{g-lawson,ms-picard}:
\begin{equation} \label{eq:picdescent}
\pic(A) \simeq \tau_{\geq 0}\pic(B)^{hG} \text{ and } \PIC(A) \simeq
\PIC(B)^{hG}.
\end{equation}
Here, $\tau_{\geq 0}$ denotes the connective cover of a spectrum. In general,
the extension $B$ is easier to understand than $A$, for instance in the case
of the $C_2$-Galois extension $\KO \ra \KU$,  one obtains information about
 $\pic(A)$  using the homotopy fixed point spectral sequence
$$ H^{-s}(G; \pi_t\pic(B)) \Rightarrow \pi_{t-s}(\pic(B)^{hG}).$$

In \cite[\S 6]{hm}  for instance, Hill and Meier use
Galois descent to determine the Picard groups of $\TMF_0(3)$ and
$\Tmf_0(3)$: 
\begin{thm} \, \cite[Theorems 6.9, 6.12]{hm} \, 
$$\Pic(\TMF_0(3)) \cong \Z/48\Z \text{ and } \Pic(\Tmf_0(3)) \cong \Z
\oplus \Z/8\Z.$$
\end{thm}

Hopkins-Mahowald-Sadofsky started the investigation of the Picard groups of
the $K(n)$-local homotopy categories for varying $n$ \cite{hms}. They
denote these Picard groups by $\Pic_n$. Note that the relevant symmetric
monoidal product for fixed $n$ is 
$$ X \otimes Y = L_{K(n)}(X \wedge Y)$$
for $K(n)$-local $X$ and $Y$. They determined $\Pic_1$ for all primes $p$:
\begin{thm} \, \cite[Theorem 3.3, Proposition 2.7]{hms} \, 
\begin{itemize}
\item
At the prime $2$: $\Pic_1 \cong \Z_2^\times \times \Z/4\Z$.
\item
For all odd primes $p$, $\Pic_1 \cong \Z_p \times \Z/q\Z$ with $q=2p-2$.
\end{itemize}
\end{thm}

In the $K(n)$-local setting the notion of algebraic elements in $\Pic_n$ 
is slightly more involved. Hopkins, Mahowald and Sadofsky show \cite{hms} 
(see also \cite[Theorem 2.4]{ghmr}) that a $K(n)$-local spectrum $X$ is 
$K(n)$-locally invertible if and only if $\pi_*(L_{K(n)}(E_n \wedge X))$ 
is a free $(E_n)_*$-module of rank one and if and only if 
$\pi_*(L_{K(n)}(E_n \wedge X))$ is invertible as a continuous module over the 
completed group ring $(E_n)_*[[\mathbb{G}_n]]$. Here, $\mathbb{G}_n$ is 
the full Morava-stabilizer group. Hence applying 
$\pi_*(L_{K(n)}(E_n \wedge -)$ gives a map from $\Pic_n$ to the Picard group 
of  continuous  $(E_n)_*[[\mathbb{G}_n]]$-modules and this group is called 
$\Pic_n^{\text{alg}}$. The kernel of the map, $\kappa_n$, collects the 
exotic elements in $\Pic_n$: 
$$ 0 \ra \kappa_n \ra \Pic_n \ra \Pic_n^{\text{alg}}.$$

For odd primes, all elements in $\Pic_1$ can be detected algebraically
but for $p=2$ one has a non-trivial element in $\kappa_1$. See \cite{ghmr}  
for $\Pic_2$ at $p=3$ and a general overview. There is ongoing work on 
$\Pic_2$ at $p=2$ by Agn\`es Beaudry, Irina Bobkova, Paul Goerss and 
Hans-Werner Henn. 

\subsection{Brauer groups of commutative rings}
Probably most of you will know the definition of the Brauer group of a field.
But as for many features that we want to transfer to the spectral world we
need to consider algebraic concepts developed for commutative rings
(not fields).

Azumaya started to think about general Brauer groups \cite{azumaya} in the
setting of local rings. A
general definition of the Brauer group of a commutative ring $R$ was
given by Auslander and Goldman \cite{auslander-g} as Morita equivalence
classes of Azumaya algebras. The Brauer group was then
globalized
to schemes by Grothendieck \cite{grothendieck}. He also shows that the
Brauer group of the initial ring $\Z$ is trivial; this is a byproduct of his
identification of Brauer groups of number rings in
\cite[III, Proposition (2.4)]{grothendieck}.

\subsection{Brave new Brauer groups}
Baker and Lazarev define in \cite{bl} what an Azumaya algebra spectrum is.
We use one version of this definition in \cite{BRS} to develop
Brauer groups for commutative ring spectra. Related concepts can be found in
\cite{johnson} and \cite{toen}.

Fix a cofibrant commutative ring spectrum $R$.
\begin{defn}
A cofibrant associative $R$-algebra $A$ is called an
\emph{Azumaya $R$-algebra spectrum} if $A$ is dualizable and faithful as an
$R$-module spectrum and if the canonical map
$$ A \wedge_R A^o \ra F_R(A, A)$$
is a weak equivalence.
\end{defn}

We list some crucial properties of Azumaya algebra spectra. For the first
property recall the discussion of derived centers from Definition
\ref{def:derivedcenter}.
\begin{prop}
\begin{enumerate}
\item[]
\item \cite[Proposition 2.3]{bl}
If $A$ is an Azumaya $R$-algebra spectrum, then $A$ is homotopically central
over $R$, \ie, $R \ra \THH_R(A)$ is a weak equivalence.
\item \cite[Proposition 1.3]{BRS}
Every Azumaya $R$-algebra spectrum $A$ is separable over $R$.
\item \cite[Proposition 1.5]{BRS}
If $A$ is Azumaya over $R$ and if $C$ is a cofibrant commutative $R$-algebra
then $A \wedge_R C$ is Azumaya over $C$. Conversely, if $C$ is as above and
dualizable and faithful as an $R$-module, then $A \wedge_R C$ being Azumaya
over $C$ implies that $A$ is Azumaya over $R$.

If $A$ and $B$ are Azumaya over $R$, then $A \wedge_R B$ is also Azumaya over
$R$.
\item \cite[2.2]{BRS}
If $M$ is a faithful, dualizable, cofibrant $R$--module, then
(a cofibrant replacement of) $F_R(M, M)$ is an $R$-Azumaya algebra spectrum.
\end{enumerate}
\end{prop}
Thus the endomorphism Azumaya algebras are the ones that are always there
and you want to ignore them.
\begin{defn}
Let $A$ and $B$ be two Azumaya $R$-algebra spectra. We call them 
\emph{Brauer equivalent} if there are dualizable, faithful $R$-modules
$N$ and $M$ 
such that there is an $R$-algebra equivalence
$$ A \wedge_R F_R(M, M) \simeq B \wedge_R F_R(N, N).$$

We denote by $Br(R)$ the set of Brauer equivalence classes of $R$-Azumaya
algebra spectra.
\end{defn}
Note that $Br(R)$ is an abelian group with multiplication induced by the
smash product over $R$. Johnson shows \cite[Lemma 5.7]{johnson}
that one can reduce the above relation to what he calls
\emph{Eilenberg-Watts
equivalence}. This implies that one can still think about the Brauer group of
a commutative ring spectrum as the Morita equivalence classes of
Azumaya algebra spectra.

We showed a Galois descent result
\cite[Proposition 3.2]{BRS}, saying that under some natural condition you
can descent an Azumaya algebra $C$ over $B$ to an Azumaya algebra $C^{hG}$
over $A$ is $A \ra B$ is a faithful $G$-Galois extension with finite Galois
group $G$.

\subsection{Examples of Brauer groups}
As we know that $Br(\Z)=0$, we conjecture \cite{BRS} that the
Brauer group of the initial ring spectrum is also trivial. This
conjecture was proven in
\cite[Corollary 7.17]{antieau-g}. They actually showed a much stronger
result:
\begin{thm} \, \cite[Theorem 7.16]{antieau-g} \, \label{thm:brconn}
If $R$ is a connective commutative ring spectrum such that $\pi_0(R)$ is
either $\Z$ or the Witt vectors $W(\F_q)$, then the Brauer group of $R$
is trivial.
\end{thm}

Different approaches can be used to construct a \emph{Brauer space}
for a commutative 
ring spectrum $R$, $Br_R$, \cite[Definition 7.1]{antieau-g},
\cite[\S 5]{g-lawson}, \cite{szymik} and to 
show that this space is a delooping of the Picard space,
$\PIC$ 
$$ \Omega Br_R \simeq \PIC(R)$$
with $\pi_0(Br_R) \cong Br(R)$.

An important question in the classical context of Brauer groups of schemes is 
to which extend these groups can be controlled by the second \'etale cohomology 
group. See the introduction of \cite{toen} for a nice overview. 
To\"en shows that for quasi-compact and quasi-separated schemes $X$ 
one can identify the \emph{derived Brauer group of $X$} with 
$H^1_{\text{\'et}}(X; \mathbb{G}_m) \times H^2_{\text{\'et}}(X; \mathbb{G}_m)$. 
The work of Antieau and Gepner \cite[\S 7.4]{antieau-g}  
relates Brauer groups of connective commutative ring spectra to
\'etale cohomology 
groups by establishing a spectral sequence starting from \'etale 
cohomology groups for \'etale sheaves over a connective commutative 
ring spectrum converging to the homotopy groups of the Brauer space 
\cite[Theorem 7.12]{antieau-g}. 

It is not hard to see that the integral version of the quaternions gives a
non-trivial element in $Br(S[\frac{1}{2}])$ \cite[Proposition 6.3]{BRS};
Antieau and Gepner show \cite[Corollary 7.18]{antieau-g}
$$ Br(S\textstyle{[\frac{1}{p}]}) \cong \Z/2\Z \quad \text{ for all primes } p$$
and they prove the existence of a short exact sequence
$$ 0 \ra Br(S_{(p)}) \ra \Z/2\Z \oplus \bigoplus_{q \neq p} \Q/\Z \ra 
\Q/\Z \ra 0$$ 
by applying \cite[Corollary 7.13]{antieau-g} where they calculate the homotopy 
groups
of the Brauer space of any connective commutative ring spectrum $R$ in terms
of \'etale cohomology groups and the homotopy groups of $R$.

They use the classical exact sequence for the Brauer group 
of the rationals \cite[\S 2]{grothendieck} coming from the 
Albert-Brauer-Hasse-Noether theorem:  
$$ 0 \ra Br(\Q) \ra \Z/2\Z \oplus \bigoplus_{p \text{ prime }} Br(\Q_p) 
\ra \Q/\Z \ra 0$$
with $Br(\Q_p) =\Q/\Z$. 
This determines $Br(\Z[\frac{1}{p}])$ and $Br(\Z_{(p)})$ and this in turn 
gives the above result for the sphere spectra with $p$ inverted or 
localized at $p$.

In \cite[Theorem 10.1]{BRS} we show that the $K(n)$-local Brauer group of the 
$K(n)$-local sphere is non-trivial at least for odd primes and $n >1$. 

Gepner and Lawson prove a version of Galois descent for a suitable
$\infty$-category of Azumaya algebras:
\begin{thm} \, \cite[Theorem 6.15]{g-lawson} \,
There is an equivalence of symmetric monoidal
$\infty$-categories
$$ Az_A \ra (Az_B)^{hG}$$
for every $G$-Galois extension $A \ra B$ with finite $G$.
\end{thm}
They also construct
a map of $\infty$-groupoids $Az_R \ra Br_R$ for any commutative ring spectrum
$R$ and show that this map is essentially surjective, such that equality in
$\pi_0(Br_R)$ corresponds precisely to Morita equivalence. They investigate
the algebraic Brauer groups (\ie, the Morita classes of Azumaya
algebras over the coefficients) \cite[\S 7.1]{g-lawson} of $2$-periodic
commutative ring spectra with vanishing odd homotopy groups, such as
$\KU$ or $E_n$, by relating them to the classical Brauer-Wall group of
$\pi_0$ of the ring spectrum and they identify a non-trivial Morita
class of a quaternion $\KO$-algebra that becomes Morita-trivial over $\KU$.

There is recent work by Hopkins and Lurie \cite{hopkins-l} who identify
the
$K(n)$-local Brauer group of a Lubin-Tate spectrum $E$ at all primes. For
odd primes they obtain:
\begin{thm} \, \cite[Theorem 1.0.11]{hopkins-l} \, 
The $K(n)$-local Brauer group of $E$ is the product of the
Brauer-Wall group of the residue field $\pi_0(E)/m$ and a group $Br'(E)$ which
in turn can be expressed as an inverse limit of abelian groups $Br'_\ell$ such
that the kernel of $Br'_\ell \ra Br'_{\ell-1}$ is non-canonically isomorphic
to $m^{\ell+2}/m^{\ell+3}$.
\end{thm}
One ingredient is their construction of 
\emph{atomic $E$-algebra spectra} \cite[Definition 1.0.2]{hopkins-l} 
via a Thom spectrum construction relative to $E$ for polarizations of lattices 
\cite[Definition 3.2.1]{hopkins-l} using the machinery from 
\cite{abghr,abghr2}. Here, the starting point is a lattice $\Lambda$ of finite 
rank together with a \emph{polarization map} $Q\colon K(\Lambda, 1) \ra \PIC(E) 
\simeq \Pic(E) \times BGL_1(E)$.


\begin{thebibliography}{999}

\bibitem{abghr}
Matthew Ando, Andrew J.~Blumberg, David Gepner, Michael J.~Hopkins,
Charles Rezk, Units of ring spectra, orientations and Thom spectra via
rigid infinite loop space theory. J. Topol. 7 (2014), no. 4,
1077--1117.

\bibitem{abghr2}
Matthew Ando, Andrew J.~Blumberg, David Gepner, Michael J.~Hopkins,
Charles Rezk, An $\infty$-categorical approach to $R$-line bundles, $R$-module
Thom spectra, and twisted $R$-homology, J. Topol. 7 (2014), no. 3, 869--893.

\bibitem{ahr}
Matthew Ando, Michael J.~Hopkins, Charles Rezk,
Multiplicative orientations of $\KO$-theory and the spectrum of topological
modular forms, available at
\url{http://www.math.uiuc.edu/~rezk/papers.html}

\bibitem{angeltveit}
Vigleik Angeltveit, Topological Hochschild homology and cohomology of
$A_\infty$ ring spectra,  Geom. Topol. 12 (2008), no. 2, 987--1032.

\bibitem{ahl}
Vigleik Angeltveit, Michael A.~Hill, Tyler Lawson, Topological
Hochschild homology of $\ell$ and $ko$,  Amer. J. Math. 132 (2010), no. 2,
297--330.


\bibitem{antieau-g}
Benjamin Antieau, David Gepner, Brauer groups and \'etale cohomology in
derived algebraic geometry. Geom. Topol. 18 (2014), no. 2,
1149--1244.

\bibitem{acb}
Omar Antol\'in-Camarena, Tobias Barthel,
A simple universal property of Thom ring spectra, arXiv:1411.7988. 

\bibitem{auslander-g}
Maurice Auslander, Oscar Goldman, The Brauer group of a commutative
ring. Trans. Amer. Math. Soc. 97,  (1960),  367--409.


\bibitem{ausoni}
Christian Ausoni, Topological Hochschild homology of connective
complex K-theory, Amer. J. Math. 127, (2005), 1261--1313.

\bibitem{azumaya}
Gor\^{o} Azumaya, On maximally central algebras,
Nagoya Math. J. 2, (1951), 119--150.


% \bibitem{BDRR}
% Nils A.~Baas, Bj{\o}rn Ian Dundas, Birgit Richter, John  Rognes,
% Ring completion of rig categories. J. Reine Angew. Math. 674 (2013),
% 43--80.



\bibitem{bl}
Andrew Baker, Andrey Lazarev,
Topological Hochschild cohomology and generalized Morita
equivalence. Algebr. Geom. Topol. 4 (2004), 623--645.

\bibitem{br}
Andrew Baker, Birgit Richter,
On the $\Gamma$-cohomology of rings of numerical polynomials and
$E_\infty$ structures on $K$-theory. Comment. Math. Helv. 80 (2005),
no. 4, 691--723.

\bibitem{br-picard}
Andrew Baker, Birgit Richter,
Invertible modules for commutative $S$-algebras with residue
fields,
Manuscripta Math. 118 (2005), no. 1, 99--119.

\bibitem{br-covers}
Andrew Baker, Birgit Richter,
Uniqueness of $E_\infty$ structures for connective covers, Proc. Amer. Math.
Soc. 136 (2008), no. 2, 707--714.

\bibitem{BRS}
Andrew Baker, Birgit Richter, Markus Szymik, Brauer groups for
commutative $S$-algebras. J. Pure Appl. Algebra 216 (2012), no. 11,
2361--2376.

\bibitem{basterra}
Maria Basterra, Andr\'e-Quillen cohomology of commutative $S$-algebras. J.
Pure Appl. Algebra 144 (1999), no. 2, 111--143.

\bibitem{bmthom}
Maria Basterra, Michael A.~Mandell,
Homology and cohomology of $E_\infty$ ring spectra,
Math. Z. 249 (2005), no. 4, 903--944.


\bibitem{bama}
Maria Basterra, Michael A.~Mandell,
The multiplication on $\BP$,
J. Topol. 6 (2013), no. 2, 285--310.

\bibitem{BaRi}
Maria Basterra, Birgit Richter,
(Co-)homology theories for commutative ($S$-)algebras. Structured ring
spectra, London Math. Soc. Lecture Note Ser., 315, Cambridge
Univ. Press, Cambridge, 2004, 115--131.

\bibitem{BSS}
Samik Basu, Steffen Sagave, Christian Schlichtkrull,
Generalized Thom spectra and their topological Hochschild homology,
arXiv:1608.08388.

\bibitem{bayindir}
Haldun \"Ozg\"ur Bay{\i}nd{\i}r,
Topological Equivalences of E-infinity Differential Graded Algebras,
arXiv:1706.08554.

\bibitem{bcs}
Andrew Blumberg, Ralph Cohen, Christian Schlichtkrull,
Topological Hochschild homology of Thom spectra and the free loop space,
Geom. Topol. 14, (2010), 1165--1242.

\bibitem{bv}
J.~Michael Boardman, Rainer M.~Vogt,
Homotopy-everything H-spaces,
Bull. Amer. Math. Soc. 74, (1968),  1117--1122.

\bibitem{boe-thh}
Marcel B\"okstedt, Topological Hochschild homology, preprint. 

\bibitem{Boe}
Marcel B\"okstedt, The topological Hochschild homology of $\Z$ and of
$\Z/p\Z$, preprint.

\bibitem{boemad}
Marcel B\"okstedt, Ib Madsen,
Topological cyclic homology of the integers. $K$-theory (Strasbourg,
1992). Ast\'erisque No. 226 (1994), 7--8, 57--143.


\bibitem{BP}
Edgar H.~Brown, Franklin P.~Peterson,
A spectrum whose $\mathbb{Z}_p$ cohomology is the algebra of reduced $p$th
powers, Topology,  1966, 149--154.

\bibitem{hinfty}
Robert R.~Bruner, J.~Peter May, James E.~McClure, Mark Steinberger,
$H_\infty$ ring spectra and their applications,
Lecture Notes in Mathematics, 1176. Springer-Verlag, Berlin, 1986.
viii+388 pp.

\bibitem{cm}
Steven Greg Chadwick, Michael A.~Mandell, $E_n$-genera, 
Geom. Topol. 19 (2015), no. 6, 3193--3232. 

\bibitem{cmnn}
Dustin Clausen, Akhil Mathew, Niko Naumann, Justin Noel,
Descent in algebraic $K$-theory and a conjecture of Ausoni-Rognes,
arXiv:1606.03328.



\bibitem{DS}
Daniel Dugger, Brooke Shipley, Topological equivalences for
differential graded algebras. Adv. Math. 212 (2007), no. 1, 37--61.


\bibitem{dlr}
Bj{\o}rn Ian Dundas, Ayelet Lindenstrauss, Birgit Richter, On higher
topological Hochschild homology of rings of integers,
arXiv:1502.02504, to appear in Math. Res. Letters.


\bibitem{dlr2}
Bj{\o}rn Ian Dundas, Ayelet Lindenstrauss, Birgit Richter,
Towards an understanding of ramified extensions of structured ring
spectra, arxiv:1604.05857.

\bibitem{EKMM}
Anthony D.~Elmendorf, Igor Kriz, Michael A.~Mandell, J.~Peter May,
Rings, modules, and algebras in stable homotopy theory. With an appendix by
M. Cole. Mathematical Surveys and Monographs, 47. American
Mathematical Society, Providence, RI, (1997), {\rm xii}+249 pp.

\bibitem{e-mandell}
Anthony D.~Elmendorf, Michael A.~Mandell,
Rings, modules, and algebras in infinite loop space
theory, Adv. Math. 205 (2006), no. 1, 163--228.

\bibitem{g-lawson}
David Gepner, Tyler Lawson, Brauer groups and Galois cohomology of
commutative ring spectra, arxiv:1607.01118.

\bibitem{ghmr}
Paul G.~Goerss, Hans-Werner Henn, Mark Mahowald, Charles Rezk, 
On Hopkins' Picard groups for the prime $3$ and chromatic level $2$, 
J. Topol. 8 (2015), no. 1, 267--294.

\bibitem{GH}
Paul G.~Goerss, Michael J.~Hopkins,
Moduli spaces of commutative ring spectra. Structured ring spectra,
 London Math. Soc. Lecture Note Ser., 315, Cambridge Univ. Press,
Cambridge, 2004, 151--200.

\bibitem{GHall}
Paul G.~Goerss, Michael J.~Hopkins,
Moduli problems for structured ring spectra, available at
\url{http://www.math.northwestern.edu/~pgoerss/}.

\bibitem{goodwillie}
Thomas G.~Goodwillie,
Cyclic homology, derivations, and the free loopspace,
Topology 24 (1985), no. 2, 187--215.

\bibitem{greenlees}
John P.~C.~Greenlees,
Spectra for commutative algebraists, Interactions between homotopy
theory and algebra,
Contemp. Math., 436, Amer. Math. Soc., Providence, RI, 2007, 149--173.

\bibitem{greenlees-go}
John P.~C.~Greenlees, 
Ausoni-B\"okstedt duality for topological Hochschild homology,  
J. Pure Appl. Algebra 220 (2016), no. 4, 1382--1402. 

\bibitem{grothendieck}
Alexander Grothendieck, Le groupe de Brauer I--III. In: Dix expos\'es
sur la cohomologie des sch\'emas,  Adv. Stud. Pure Math., 3,
North-Holland, Amsterdam, 1968, 46--188.


\bibitem{hms-picard}
Drew Heard, Akhil Mathew, Vesna Stojanoska,
Picard groups of higher real K-theory spectra at height $p-1$,
Compositio Mathematica 153, (2017), 1820--1854.

\bibitem{hesselholt-m}
Lars Hesselholt, Ib Madsen,
On the K-theory of finite algebras over Witt vectors of perfect fields,
Topology 36 (1997), no. 1, 29--101.

\bibitem{hl}
Michael A.~Hill, Tyler Lawson,
Automorphic forms and cohomology theories on Shimura curves of small
discriminant, Adv. Math. 225 (2010), no. 2, 1013--1045.

\bibitem{hm}
Michael A.~Hill, Lennart Meier, The $C_2$-spectrum $\mathit{Tmf}_1(3)$
and its invertible modules,  Algebraic and Geometric Topology 17,
(2017), 1953--2011.

\bibitem{hirschhorn}
Philip S.~Hirschhorn,
Model categories and their localizations,
Mathematical Surveys and Monographs, 99,
American Mathematical Society, Providence, RI, (2003), xvi+457 pp.


\bibitem{hopkins-lawson}
Michael J.~Hopkins, Tyler Lawson, 
Strictly commutative complex orientation theory, 
arXiv:1603.00047. 


\bibitem{hopkins-l}
Michael J.~Hopkins, Jacob Lurie,
On Brauer Groups of Lubin-Tate Spectra I, preprint available at
\url{http://www.math.harvard.edu/~lurie/}

\bibitem{hms} Michael J.~Hopkins, Mark Mahowald, Hal Sadofsky,
Constructions of elements in Picard groups. Topology and representation
theory (Evanston, IL, 1992), Contemp. Math., 158, Amer. Math.
Soc., Providence, RI, (1994), 89--126.

\bibitem{hovey}
Mark Hovey, Model categories, Mathematical Surveys and Monographs,
63. American Mathematical Society, Providence, RI, 1999. xii+209 pp.

\bibitem{hovey-ideals}
Mark Hovey, Smith ideals of structured ring spectra, arXiv:1401.2850.

\bibitem{HSS}
Mark Hovey, Brooke Shipley, Jeff Smith,
Symmetric spectra. J. Amer. Math. Soc. 13 (2000), no. 1, 149--208.

\bibitem{iyengar}
Srikanth Iyengar, Andr\'e-Quillen homology of commutative algebras,
Interactions between homotopy theory and algebra, Contemp. Math., 436,
Amer. Math. Soc., Providence, RI (2007), 203--234.

\bibitem{joachim}
Michael Joachim,
A symmetric ring spectrum representing $\KO$-theory. Topology 40 (2001),
no. 2, 299--308.

\bibitem{johnson}
Niles Johnson, Azumaya objects in triangulated
bicategories. J. Homotopy Relat. Struct. 9 (2014), no. 2, 465--493.

\bibitem{johnson-noel}
Niles Johnson, Justin Noel,
Lifting homotopy $T$-algebra maps to strict maps,
Adv. Math. 264 (2014), 593--645.

\bibitem{klang}
Inbar Klang, The factorization theory of Thom spectra and twisted
non-abelian Poincar\'e duality, arXiv:1606.03805.

\bibitem{kuhn}
Nicholas J.~Kuhn,
The McCord model for the tensor product of a space and a commutative ring
spectrum, Categorical decomposition techniques in algebraic topology
(Isle of Skye, 2001),
Progr. Math., 215, Birkh\"auser, Basel, (2004), 213--236.

\bibitem{lawson}
Tyler Lawson,
Secondary power operations and the Brown-Peterson spectrum at the
prime $2$,  arXiv:1703.00935.

\bibitem{lawson-gamma}
Tyler Lawson,
Commutative $\Gamma$-rings do not model all commutative ring
 spectra. Homology Homotopy Appl. 11 (2009), no. 2, 189--194.

\bibitem{ln1}
Tyler Lawson, Niko Naumann, Commutativity conditions for truncated
Brown-Peterson spectra of height 2. J. Topol. 5 (2012), no. 1,
137--168.

%\bibitem{ln2}
%Tyler Lawson, Niko Naumann, Strictly commutative realizations of
%diagrams over the Steenrod algebra and topological modular forms at
%the prime $2$. Int. Math. Res. Not. IMRN 2014, no. 10, 2773--2813.

\bibitem{lms}
L.~Gaunce Lewis Jr., J.~Peter May, Mark Steinberger, James E.~McClure,
Equivariant stable homotopy theory. With contributions by
J. E. McClure. Lecture Notes in Mathematics, 1213. Springer-Verlag,
Berlin, 1986. {\rm x}+538 pp.

\bibitem{lind}
John A.~Lind,
Diagram spaces, diagram spectra and spectra of units,
Algebr. Geom. Topol. 13 (2013), no. 4, 1857--1935.

\bibitem{lm}
Ayelet Lindenstrauss, Ib Madsen,
Topological Hochschild homology of number rings,
Trans. Amer. Math. Soc. 352 (2000), no. 5, 2179--2204.

\bibitem{loday}
Jean-Louis Loday, Cyclic homology, Appendix E by Mar\'ia
O. Ronco, Second edition. Chapter 13 by the author in collaboration
with Teimuraz Pirashvili. Grundlehren der Mathematischen
Wissenschaften 301, Springer-Verlag, Berlin, (1998), {\rm xx}+513 pp.

\bibitem{lurie-dagvii}
Jacob Lurie, Derived Algebraic Geometry VII: Spectral Schemes, available at
\url{http://www.math.harvard.edu/~lurie/}.

\bibitem{lurie-HA}
Jacob Lurie, Higher Algebra, available at
\url{http://www.math.harvard.edu/~lurie/}.

\bibitem{lydakis}
Manos Lydakis,
Smash products and $\Gamma$-spaces, Math. Proc. Cambridge Philos. Soc. 126
(1999), no. 2, 311--328.

\bibitem{mahowald}
Mark Mahowald, Ring spectra which are Thom complexes, 
Duke Math. J. 46 (1979), no. 3, 549--559.



\bibitem{mandell-cochains}
Michael A.~ Mandell,
$E_\infty$-algebras and $p$-adic homotopy theory,
Topology 40 (2001), no. 1, 43--94.

\bibitem{mandell-aqeinfty}
Michael A.~ Mandell,
Topological Andr\'e-Quillen cohomology and $E_\infty$
Andr\'e-Quillen cohomology, Adv. Math. 177 (2003), no. 2, 227--279.

\bibitem{mandell-integral}
Michael A.~Mandell, Cochains and homotopy
type,  Publ. Math. Inst. Hautes \'Etudes Sci. No. 103,  (2006), 213--246.

\bibitem{mandellderived}
Michael A.~ Mandell,
The smash product for derived categories in stable homotopy
theory. Adv. Math. 230 (2012), no. 4-6, 1531--1556.

\bibitem{MMSS}
Michael A.~Mandell, J.~Peter May, Stefan Schwede, Brooke Shipley,
Model categories of diagram spectra. Proc. London Math. Soc. (3) 82 (2001),
no. 2, 441--512.

\bibitem{Mat}
Akhil Mathew, $\mathit{THH}$ and base-change for Galois extensions of
ring spectra, Algebr. Geom. Topol. 17 (2017), no. 2, 693--704.

\bibitem{mm}
Akhil Mathew, Lennart Meier,
Affineness and chromatic homotopy theory,
J. Topol. 8 (2015), no. 2, 476--528.


\bibitem{mnn}
Akhil Mathew, Niko Naumann, Justin Noel,
On a nilpotence conjecture of J.~ P.~May,  J. Topol. 8 (2015), no. 4, 917--932.


\bibitem{ms-picard}
Akhil Mathew, Vesna Stojanoska,
The Picard group of topological modular forms via descent theory,
Geometry \& Topology 20 (2016),  3133--3217.



\bibitem{may-einfty}
J.~Peter May, $E_\infty$ ring spaces and $E_\infty$ ring spectra. With
contributions by Frank Quinn, Nigel Ray, and J{\o}rgen
Tornehave. Lecture Notes in Mathematics, Vol. 577. Springer-Verlag,
Berlin-New York, 1977. 268 pp.

\bibitem{may-perm}
J.~Peter May, The spectra associated to permutative
categories. Topology 17 (1978), no. 3, 225--228.

\bibitem{may-mult}
J.~Peter May, Multiplicative infinite loop space theory. J. Pure
Appl. Algebra 26 (1982), no. 1, 1--69.

\bibitem{may-biperm}
J.~Peter May, The construction of $E_\infty$ ring spaces from bipermutative
categories. New topological contexts for Galois theory and algebraic
geometry (BIRS 2008), Geom. Topol. Monogr., 16,
Geom. Topol. Publ., Coventry (2009), 283--330.

\bibitem{mm-hkr}
Randy McCarthy, Vahagn Minasian, HKR theorem for smooth $S$-algebras,
J. Pure Appl. Algebra 185 (2003), no. 1-3, 239--258.

\bibitem{MSV}
James E.~McClure, Roland Schw\"anzl, Rainer M.~Vogt, $\THH(R)\cong
R\otimes S^1$ for $E_\infty$ ring spectra. J. Pure Appl. Algebra 121
(1997), no. 2, 137--159.

\bibitem{mcc-smith}
James E.~McClure, Jeffrey H.~Smith,
A solution of Deligne's Hochschild cohomology conjecture,
Recent progress in homotopy theory (Baltimore, MD, 2000),
Contemp. Math., 293, Amer. Math. Soc., Providence, RI, (2002), 153--193.


\bibitem{mccs}
James E.~McClure, Ross E.~Staffeldt, On the topological Hochschild homology of
bu. I. Amer. J. Math. 115 (1993), no. 1, 1--45.

\bibitem{minasian}
Vahagn Minasian, Andr\'e-Quillen spectral sequence for $\THH$,
Topology Appl. 129 (2003), no. 3, 273--280.

\bibitem{moore}
John C.~Moore, Semi-simplicial complexes and Postnikov systems,  1958
Symposium internacional de topolog\'ia algebraica (International
symposium on algebraic topology),  Universidad Nacional
Aut\'onoma de M\'exico y la UNESCO, Mexico City, pp. 232--247.
Available at \url{www.maths.ed.ac.uk/~aar/papers/mexico.pdf}

\bibitem{noel}
Justin Noel,
The $T$-algebra spectral sequence: comparisons and applications,
Algebr. Geom. Topol. 14 (2014), no. 6, 3395--3417.

\bibitem{pira}
Teimuraz Pirashvili,
Hodge decomposition for higher order Hochschild homology,
Ann. Sci. \'Ecole Norm. Sup. (4) 33 (2000), no. 2, 151--179.

\bibitem{piri}
Teimuraz Pirashvili, Birgit Richter,
Robinson-Whitehouse complex and stable homotopy,
Topology 39 (2000), no. 3, 525--530.


\bibitem{Q-qhtp}
Daniel Quillen, Rational homotopy theory. Ann. of Math. (2) 90 (1969),
205--295.

\bibitem{quillen}
Daniel Quillen, On the (co-) homology of commutative rings,
Applications of Categorical Algebra (Proc. Sympos. Pure Math.,
Vol. XVII, New York, 1968)  Amer. Math. Soc., Providence,
R.I. (1970), 65--87.

%\bibitem{ravenelcob} Douglas C.~Ravenel, Complex cobordism and
%  stable homotopy groups of spheres. Pure and Applied Mathematics,
%  121. Academic Press, Inc., Orlando, FL, 1986. xx+413 pp.

\bibitem{ravenelnil}
Douglas C.~Ravenel,
Nilpotence and periodicity in stable homotopy theory. Appendix C by
Jeff Smith. Annals of Mathematics Studies, 128. Princeton University Press,
Princeton, NJ, (1992), xiv+209 pp.

\bibitem{rezk}
Charles Rezk,
Notes on the Hopkins-Miller theorem,
Homotopy theory via algebraic geometry and group representations
(Evanston, IL, 1997), Contemp. Math., 220, Amer. Math. Soc.,
Providence, RI, 1998, 313--366.

%%%
\bibitem{rs}
Birgit Richter, Brooke Shipley, An algebraic model for commutative
$\HZ$-algebras, Algebraic \& Geometric Topology 17 (2017), 2013--2038.

\bibitem{roderived}
Alan Robinson,
The extraordinary derived category.
Math. Z. 196 (1987), no. 2, 231--238.


\bibitem{roainfty}
Alan Robinson,
Obstruction theory and the strict associativity of Morava
K-theories. Advances in homotopy theory (Cortona, 1988),
London Math. Soc. Lecture Note Ser., 139, Cambridge Univ. Press,
Cambridge, 1989, 143--152.

\bibitem{robinson}
Alan Robinson, Gamma homology, Lie representations and $E_\infty$
multiplications. Invent. Math. 152 (2003), no. 2, 331--348.

\bibitem{RobWh}
Alan Robinson, Sarah Whitehouse, Operads and $\Gamma$-homology of
commutative rings. Math. Proc. Cambridge Philos. Soc. 132 (2002),
no. 2, 197--234.

%%%
\bibitem{rognestrace}
John Rognes,
Trace maps from the algebraic $K$-theory of the integers (after Marcel
B\"okstedt). J. Pure Appl. Algebra 125 (1998), no. 1-3, 277--286.


\bibitem{rognes} John Rognes, Galois extensions of structured ring
  spectra. Stably dualizable groups. Mem. Amer. Math. Soc. 192 (2008),
  no. 898, viii+137 pp.

\bibitem{rss}
John Rognes, Steffen Sagave, Christian Schlichtkrull,
Localization sequences for logarithmic topological Hochschild homology,
Math. Ann. 363 (2015), no. 3-4, 1349--1398.

\bibitem{sagave}
Steffen Sagave,
Logarithmic structures on topological K-theory spectra,
Geom. Topol. 18 (2014), no. 1, 447--490.

\bibitem{sagave-sch}
Steffen Sagave, Christian Schlichtkrull,
Diagram spaces and symmetric spectra,
Adv. Math. 231 (2012), no. 3-4, 2116--2193.

\bibitem{ssthom}
Steffen Sagave, Christian Schlichtkrull,
Virtual vector bundles and graded Thom spectra,
arXiv:1410.4492.

\bibitem{schlichtkrull}
Christian Schlichtkrull,
Higher topological Hochschild homology of Thom spectra,
J. Topol. 4 (2011), no. 1, 161--189.

\bibitem{schwede-gamma}
Stefan Schwede,
Stable homotopical algebra and $\Gamma$-spaces,
Math. Proc. Cambridge Philos. Soc. 126 (1999), no. 2, 329--356.

\bibitem{schwede-comp}
Stefan Schwede, $S$-modules and symmetric spectra. Math. Ann. 319 (2001),
no. 3, 517--532.

\bibitem{schwede-shipleydk}
Stefan Schwede, Brooke Shipley,
Equivalences of monoidal model categories,
Algebr. Geom. Topol. 3 (2003), 287--334.

\bibitem{schwede-shipleydgmod}
Stefan Schwede, Brooke Shipley,
Stable model categories are categories of modules,
Topology 42 (2003), no. 1, 103--153.

\bibitem{catsandco}
Graeme Segal,
Categories and cohomology theories.
Topology 13 (1974), 293--312.

%\bibitem{Sen}
%Andrew Senger, On the realization of truncated Brown–Peterson spectra as
%$E_\infty$-ring spectra, unpublished.

%%%

\bibitem{shishi}
Nobuo Shimada, Kazuhisa Shimakawa, 
Delooping symmetric monoidal categories. Hiroshima Math. J. 9 (1979),
no. 3, 627--645. 

\bibitem{shipleyconvenient}
Brooke Shipley,
A convenient model category for commutative ring
spectra,  Homotopy theory: relations with algebraic geometry, group
cohomology, and algebraic $K$-theory, Contemp. Math., 346,
Amer. Math. Soc., Providence, RI, 2004, 473--483.

\bibitem{shipley}
Brooke Shipley, $\HZ$-algebra spectra are differential graded
algebras. Amer. J. Math. 129 (2007), no. 2, 351--379.

\bibitem{sullivan}
Dennis Sullivan, Infinitesimal computations in topology. Inst. Hautes
\'Etudes Sci. Publ. Math. No. 47 (1977), 269--331.

\bibitem{szymik}
 Markus Szymik,
 Brauer spaces for commutative rings and structured ring spectra,
 Manifolds and K-theory. Contemp. Math. 682, (2017), 189--208.

\bibitem{tmf}
Topological modular forms, edited by Christopher L.~Douglas, John Francis, 
Andr\'e G.~Henriques and Michael A.~Hill, Mathematical Surveys and Monographs, 
201. American Mathematical Society, Providence, RI, 2014. xxxii+318 pp.


\bibitem{toen}
Bertrand To\"en,
Derived Azumaya algebras and generators for twisted derived categories,
Invent. Math. 189 (2012), no. 3, 581--652.
%%%

\bibitem{waldhausen}
Friedhelm Waldhausen,
Algebraic $K$-theory of topological spaces. I,
Algebraic and geometric topology (Proc. Sympos. Pure Math.,
Stanford Univ., Stanford, Calif., 1976), Part 1,
Proc. Sympos. Pure Math., XXXII, Amer. Math. Soc., Providence, R.I.,
1978, 35--60.

\bibitem{wg} Charles A.~Weibel, Susan C.~Geller,
\'Etale descent for Hochschild and cyclic homology,
Comment. Math. Helv. 66 (1991), no. 3, 368--388.

\bibitem{woolfson}
Richard Woolfson,
Hyper-$\Gamma$-spaces and hyperspectra.
Quart. J. Math. Oxford Ser. (2) 30 (1979), no. 118, 229--255.


\bibitem{wuergler}
Urs W\"urgler,
Commutative ring-spectra of characteristic $2$,
Comment. Math. Helv. 61 (1986), no. 1, 33--45.




\end{thebibliography}
\end{document}